  \documentclass[12pt,letterpaper,oneside]{amsart}

  \usepackage{amsmath,amssymb,amsthm}
  \usepackage{geometry}
  \usepackage{accents}
  \usepackage{graphicx}
  \usepackage{enumerate}
  \usepackage{color}
  \usepackage{xspace,verbatim}
\usepackage{epstopdf}

\input xy
\xyoption{all}
  \usepackage{epsfig} 

  \usepackage{hyperref}

  \newcommand{\calA}{\mathcal{A}}
  
  \newcommand{\calC}{\mathcal{C}}
  \newcommand{\calD}{\mathcal{D}}
  \newcommand{\calE}{\mathcal{E}}

  \newcommand{\calT}{\mathcal{T}}
  
  \newcommand{\calV}{\mathcal{V}}

  \newcommand{\CC}{\mathbb{C}}

  \newcommand{\RR}{\mathbb{R}}

  \newcommand{\ZZ}{\mathbb{Z}}

\newcommand{\Dlab}{\calD}

\newcommand{\xdkl}{X_\f}
\newcommand{\cdkl}{\mathfrak C_{\f}}
\newcommand{\pdkl}{\psi_\f}

  \newtheorem{theorem}{Theorem}[section]
  \newtheorem{proposition}[theorem]{Proposition}
  
  \newtheorem{lemma}[theorem]{Lemma}

  \newtheorem*{theorem*}{Theorem}

  \newtheorem{introthm}{Theorem}

  \theoremstyle{definition}
  \newtheorem{definition}[theorem]{Definition}
  
  \newtheorem*{claim*}{Claim}

  \newtheorem{example}[theorem]{Example}
  
  \newtheorem*{question*}{Question}
  \newtheorem*{answer*}{Answer}
  \newtheorem*{application*}{Application}

  \theoremstyle{remark}
  \newtheorem{remark}[theorem]{Remark}
  \newtheorem*{remark*}{Remark}
  
  \newtheorem{introrem}[introthm]{Remark}

  \newcommand{\propref}[1]{Proposition~\ref{Prop:#1}}

  \DeclareMathOperator{\Mod}{Mod}

  \DeclareMathOperator{\Hom}{Hom}

  \newcommand{\Aut}{\ensuremath{\operatorname{Aut}}\xspace} 
  \newcommand{\Out}{\ensuremath{\operatorname{Out}}\xspace} 
   
  \newcommand{\sV}{{\sf V}}
  
  \newcommand{\sX}{{\sf X}}
  \newcommand{\sY}{{\sf Y}}
  
  \newcommand{\f}{{\sf f}}
  \renewcommand{\t}{{\sf t}}

  \newcommand{\w}{{\sf w}}
  
  \renewcommand{\a}{{\sf a}}

  \newcommand{\A}{\mathcal{A}}
 
 \newcommand{\C}{\mathcal{C}}
 \newcommand{\D}{\mathcal{D}}

  \newcommand{\param}{{\mathchoice{\mkern1mu\mbox{\raise2.2pt\hbox{$
  \centerdot$}}
  \mkern1mu}{\mkern1mu\mbox{\raise2.2pt\hbox{$\centerdot$}}\mkern1mu}{
  \mkern1.5mu\centerdot\mkern1.5mu}{\mkern1.5mu\centerdot\mkern1.5mu}}}

  \renewcommand{\setminus}{{\smallsetminus}}

  \newcommand{\ST}{\mathbin{\Big|}} 
  \newcommand{\from}{\colon\thinspace} 

   \newcommand{\al}{\alpha}
    
  \newcommand{\lam}{\lambda}

  \newcommand{\eps}{\varepsilon}
  \newcommand{\sig}{\sigma}

  \newcommand{\into}{\hookrightarrow}

  \def\sign{\textup{sign}}

\def\Supp{\mbox{Supp}}
     
     \def\Hom{\textup{Hom}}
\def\R{\mathbb{R}}
 \def\Z{\mathbb Z}

\begin{document}

  \title    {Digraphs and cycle polynomials for free-by-cyclic groups.}
  \author   {Yael Algom-Kfir, Eriko Hironaka
 and Kasra Rafi}
\date{\today}

\thanks{\tiny \noindent 
The second author was partially supported by a collaboration grant from the Simons Foundation \#209171,
the third author was partially supported by an NSERC Discovery grant, RGPIN \#435885}

\begin{abstract}
Let $\phi \in \Out(F_n)$ be a free group outer automorphism that can be represented by an
expanding, irreducible train-track map.  The automorphism  $\phi$ determines 
a free-by-cyclic group $\Gamma=F_n \rtimes_\phi \Z,$ and a homomorphism 
$\alpha \in H^1(\Gamma; \Z)$. By work of Neumann, Bieri-Neumann-Strebel
and Dowdall-Kapovich-Leininger, $\alpha$ has an open cone neighborhood $\A$
in $H^1(\Gamma;\R)$ whose  integral points correspond to
other  fibrations of $\Gamma$ whose associated outer automorphisms 
are themselves representable by 
expanding irreducible train-track maps. In this paper, we define an 
analog of McMullen's Teichm\"uller polynomial that computes the dilatations
of all outer automorphism in $\A$.
\end{abstract}
  \maketitle

\section{Introduction}
There is continually growing evidence of a powerful analogy between 
the mapping class group $\Mod(S)$ of a closed oriented surfaces $S$ of finite type, and 
the group of outer automorphisms  $\Out(F_n)$ of free groups $F_n$. 
A recent advance in this direction can be found in work of  Dowdall-Kapovich-Leininger 
\cite{DKL} who developed  an analog of the fibered face theory of surface 
homeomorphisms due to Thurston \cite{Thurston_Norm} and Fried 
\cite{Fried_dilatation_extends}. In this paper we develop the analogy further by 
defining a version of McMullen's Teichm\"uller
polynomial for surface automorphisms defined in \cite{McMullen_polynomial_invariants} in the setting of outer automorphisms. 

\subsection*{Fibered face theory for free-by-cyclic groups.} 
 A free-by-cyclic group
$$
\Gamma = F_n \rtimes_\phi \Z,
$$
is a semi-direct product defined by an element $\phi \in \Out(F_n)$.  
If $x_1,\dots,x_n$ are generators of $F_n$, and $\phi_\circ \in \Aut(F_n)$ is 
a representative automorphism in the class $\phi$, then $\Gamma$ has a finite 
presentation
\begin{eqnarray*}\label{presentation-eqn}
\langle x_1,\dots,x_n, s \mid s x_i s^{-1} = \phi_\circ(x_i )\quad i=1,\dots,n \rangle.
\end{eqnarray*}
There is a distinguished homomorphism $\alpha_\phi\from \Gamma \rightarrow \Z$ 
induced by projection to the second coordinate. That is, 
$\alpha_\phi$ is an element of $H^1(\Gamma; \Z)$ and $F_n$ is the kernel
of $\alpha_\phi$. 

The deformation theory of free-by-cyclic groups started with the work of 
Neumann \cite{Neu} and  Bieri-Neumann-Strebel \cite{BNS}, 
where they showed there is an open cone in $H^1(\Gamma; \R)$ containing
all primitive integral element in $H^1(\Gamma;\R)$ that have
finitely generated kernels. In \cite{DKL} Dowdall-Kapovich-Leininger showed
that this deformation can be understood geometrically in a possibly smaller cone. 

More precisely, assume $\phi \in \Out(F_n)$ is representable by an expanding 
irreducible train-track map (see  \cite{Kapovich_Algorithmic_Detection} \cite{DKL}, 
and Section~\ref{freegroupautomorphisms-sec} for definitions). The outer 
automorphism $\phi \in \Out(F_n)$ may admit many train-track representatives 
$f$ and every train-track representative can be decomposed into a sequence of folds 
$\f$ \cite{Stallings_folds}  which is also 
non-unique. Dowdall, Kapovich and Leininger showed 
(see \cite{DKL}, Theorems A):

\begin{theorem}[Dowdall-Kapovich-Leininger] \label{DKL-thm} 
For $\phi \in \Out(F_n)$ that is representable by an expanding irreducible train-track 
map and an associated folding sequence $\f$, there is an open cone neighborhood 
$\calA_\f$ of $\alpha_\phi$ in $\Hom(\Gamma;\R)$,
such that, all primitive integral elements $\alpha \in \calA$, are associated to a 
free-by-cyclic decomposition
$$
\Gamma= F_{n_\alpha} \rtimes_{\phi_\alpha} \Z
$$
where $\alpha = \alpha_{\phi_\alpha}$ and $\phi_\alpha \in \Out(F_{n_\alpha})$ is 
also representable by an expanding irreducible train-track map. 
\end{theorem}

We call $\calA_\f$ a {\it DKL-cone} associated to $\phi$.
 
\subsection*{Main Result}
Our main theorem is an analog of results in \cite{McMullen_polynomial_invariants}
in the setting of the outer automorphism groups (see below for more on the motivation
behind the result). 
For a given $\phi$, there are many DKL cones associated to $\phi$ since $A_\f$ 
depends on the choice of the train-track representative $f$ and folding sequence $\f$.
We show that there is a more unified picture. Namely, there is a cone $\calT_\phi$
depending only on $\phi$ that contains every cone $\calA_\f$. The cone $\calT_\phi$ 
is the support of a convex, real analytic, homogenous function $L$ of degree $-1$ 
whose restriction to every cone $\calA_\f$ is the logarithm of dilatation function. 
Moreover, this function can be computed via specialization of a single polynomial
$\Theta$ that also depends only on $\phi$. 

Our approach is combinatorial. We associate a labeled digraph to the splitting 
sequence $\f$.  This gives a combinatorial description of $\f$ and in turn defines a 
cycle polynomial $\theta$ and the cone $\calT_\phi$.   We analyze the effect of certain 
elementary moves on digraphs and show that their associated cycle polynomial 
and cone remain unchanged under these elementary moves. We show that as 
we pass to different fibrations of $\Gamma$ corresponding to other 
integral points of $\calA_\f$, the digraph changes by elementary moves, as do the
digraphs associated to different splitting sequences $f$.  This establishes the 
independence of $\theta$ and $\calT_\phi$ from the choice of splitting sequence $\f$.
The polynomial $\Theta$ is a factor of the cycle polynomial $\theta$ determined by
the log dilatation factor and does not depend on the choice of train track map.
  

We establish some terminology before stating the main theorem more
precisely. For $\phi \in \Out(F_n)$ that is representable by an expanding irreducible
train-track map and a nontrivial $\gamma \in F_n$, the growth rate of cyclically 
reduced word-lengths of $\phi^k(\gamma)$ is exponential, with a base 
$\lambda(\phi)>1$ that is independent of $\gamma$. The constant $\lambda(\phi)$ 
is called the {\it dilatation} (or {\it expansion factor}) of $\phi$.

Let $G$ be a finitely generated free abelian group of rank $k$, and let
$$
\theta = \sum_{g \in G} a_g g, \qquad a_g \in \Z
$$
be an element of the group ring $\Z G$.  For $\alpha \in \Hom(G;\Z)$, the 
\emph{specialization} of $\theta$ at $\alpha$ is the single variable integer polynomial
$$
\theta^{(\alpha)} (x) = \sum_{g \in G} a_g x^{\alpha(g)} \in \Z[x].
$$
The {\it house} of a $\theta^{(\alpha)} (x)$ is given by
$$
\left| \theta^{(\alpha)} \right| = 
\max \Big\{|\mu|\ \ST \mu \in \CC, \ \theta^{(\alpha)}(\mu) = 0 \Big\}.
$$

\begin{introthm}\label{main-thm} Let $\phi \in \Out(F_n)$ be an 
outer automorphism that is representable by an expanding irreducible train-track
map, $\Gamma =  F_n \rtimes_\phi \ZZ$ and let 
$G = \Gamma^{\mbox{ab}}/{\mbox{Torsion}}$. Then there exists an element 
$\Theta \in \Z G$ (well-defined up to an automorphism of $\Z G$)
with the following properties.
\begin{enumerate}
\item There is an open cone $\calT_\phi \subset  \Hom(G;\R)$ dual to a vertex of
the Newton polygon of $\Theta$ so that, for any expanding irreducible train-track representative  
$f  \from \tau \rightarrow \tau$ and any folding decomposition $\f$ of $f$, 
we have
\[
\calA_\f \subset \calT_\phi. 
\]
\item  For any integral $\alpha \in \calA_\f$, we have
$$
\big|\Theta^{(\alpha)}\big| = \lambda(\phi_\alpha).
$$
\item The function 
$$
L(\alpha) = \log \big|\Theta^{(\alpha)}\big|,
$$
which is defined on the primitive integral points of $\calA_\f$, extends to a real 
analytic, convex function on $\calT_\phi$ that is homogeneous of degree $-1$ and goes 
to infinity toward the boundary of any affine planar section of $\calT_\phi$.
\item The element $\Theta$ is minimal with respect to property (2), that is,  
if $\theta \in \Z G$ satisfies
$$
\big|\theta^{(\al)} \big| = \lambda(\phi_\al)
$$
on the integral elements of some open sub-cone of $\calT_\phi$ 
then $\Theta$ divides $\theta$.
\end{enumerate}
\end{introthm}

\begin{introrem}\label{DKLpoly-rem}
In their original paper \cite{DKL} Dowdall, Kapovich, and Leininger also show that
$\log (\lambda(\phi_\alpha))$ is convex and has degree $-1$ and in the subsequent 
paper \cite{DKL_poly}, using a different approach from ours, they give an independent 
definition of an element $\Theta_{\textup{\tiny DKL}} \in \Z G$ with the property that 
$\lambda(\phi_\alpha) = \big|\Theta_{\mbox{\tiny DKL}}^{(\alpha)}\big|$ for
 $\alpha \in A_\f$. Property (3) of Theorem~\ref{main-thm} implies that $\Theta$ 
divides $\Theta_{\mbox{\tiny DKL}}$.
\end{introrem}

\begin{introrem}
Thinking of $G$ as an abelian group generated by $t_1,\dots,t_k$, we
can identify the elements of $G$ with monomials in the symbols $t_1,\dots,t_k$, and hence $\Z G$ with Laurent
polynomials in $\Z(t_1,\dots,t_k)$.   Thus, we can associate to $\theta \in \Z G$ 
a polynomial $\theta(t_1,\dots,t_k) \in \Z(t_1,\dots,t_k)$.
Identifying $\Hom(G;\Z)$ with $\Z^k$, each element $\alpha = (a_1,\dots,a_k)$
defines a specialization of $\theta = \theta(t_1,\dots,t_k)$ by
$$
\theta^{(\alpha)}(x) = \theta(x^{a_1},\dots,x^{a_k}).
$$
For ease of notation, we mainly use the group ring notation through most of this paper.
\end{introrem}

\subsection*{Motivation from pseudo-Anosov mapping classes on surfaces.}
Let $S$ be a closed oriented surface of negative finite Euler characteristic.  
A {\it mapping class} $\phi = [\phi_\circ]$ is an isotopy class of
homeomorphisms
$$
\phi_\circ\from S \rightarrow S.
$$    
The mapping torus $X_{(S,\phi)}$
of the pair $(S,\phi)$ is the quotient space
$$
X_{(S,\phi)} = S \times [0,1]/{(x,1) \sim (\phi_\circ(x),0)}.
$$
Its homeomorphism type is independent of the choice of representative $\phi_\circ$ 
for $\phi$. The mapping torus $X_{(S,\phi)}$ has a distinguished fibration 
$\rho_\phi\from X_{(S,\phi)} \rightarrow S^1$
defined by projecting $S \times [0,1]$ to its second component and identifying endpoints.
Conversely, any fibration $\rho \from X \rightarrow S^1$ of a 3-manifold $X$ over a 
circle can be written as the mapping torus of a unique mapping class $(S,\phi)$, with 
$\rho = \rho_\phi$.  The mapping class $(S,\phi)$ is called the {\it monodromy} of $\rho$.

Thurston's fibered face theory \cite{Thurston_Norm} gives a parameterization of 
the fibrations of  a $3$--manifold $X$ over the circle with connected fibers by
the primitive integer points on a finite union of disjoint convex cones in $H^1(X;\R)$, 
called {\it fibered cones}.  The {\it dilatation} of $\phi$ and is denoted by $\lambda(\phi)$.
Thurston showed that the mapping torus of any pseudo-Anosov mapping class is 
hyperbolic, and the monodromy of any fibered hyperbolic 3-manifold is 
pseudo-Anosov.  It follows that  the set of all pseudo-Anosov mapping classes 
partitions into subsets corresponding to integral points on fibered cones of
hyperbolic 3-manifolds.

By results  of Fried \cite{Fried_dilatation_extends} (cf. \cite{Matsumoto87} 
\cite{McMullen_polynomial_invariants})  the function $\log \lambda (\phi)$ defined 
on integral points of a fibered cone $\calT$ extends to a continuous convex function 
$$
\mathcal{Y} \from \calT \rightarrow \R,
$$
that is a homogeneous of degree $-1$, and goes to infinity toward the
boundary of any affine planar section of $\calT$.  The {\it Teichm\"uller polynomial of 
a fibered cone $\calT$} defined in \cite{McMullen_polynomial_invariants} is an element  
$\Theta_{\mbox{\tiny{Teich}}}$
in the group ring $\Z G$ where $G = H_1(X;\Z)/{\mbox{Torsion}}$.
The group ring $\Z G$ can be thought of as a ring of Laurent polynomials in the generators
of $G$ considered as a multiplicative group, and in this way, we can think of 
$\Theta_{\mbox{\tiny{Teich}}}$ as a polynomial.
The Teichm\"uller polynomial $\Theta_{\mbox{\tiny{Teich}}}$ has
 the property that the dilatation $\lambda(\phi_\alpha)$ of every mapping class $\phi_\alpha$ 
 associated to an integral point $\alpha \in \calT$,
$\lambda(\phi_\al)$ can be obtained from $\Theta_{\text{\tiny{Teich}}}$ by taking the house of
the specialization.  Furthermore,  the cone $\calT$
and the function $\mathcal{Y}$ are determined by $\Theta_{\text{\tiny{Teich}}}$.
Our work is a step towards reproducing this picture in the setting
of $\Out(F_n)$. 

\subsection*{Organization of paper}

In Section~\ref{digraphs-sec}
we establish some preliminaries about Perron-Frobenius  digraphs  $D$ with edges labeled by a free abelian group $G$.
Each digraph $D$ determines a cycle complex $C_D$ and
cycle polynomial $\theta_D$ in the group ring $\Z G$.  Under certain extra conditions,
we define a cone $\calT$, which we call the McMullen cone, and show that
$$
L(\alpha) = \log |\theta_D^{(\alpha)}|
$$
defined for integral elements of $\calT$ extends to a homogeneous function of degree -1 that is
real analytic and convex on $\calT$ and goes to infinity toward the boundary of affine planar sections of $\calT$.   Furthermore, 
we show the existence of  a distinguished factor $\Theta_D$ of $\theta_D$ with the property that
$$
|\Theta_D^{(\alpha)}| = |\theta_D^{(\alpha)}|,
$$
and $\Theta_D$ is minimal with this property.
Our proof uses a key result of McMullen (see \cite{McMullen_polynomial_invariants}, Appendix A).

In Section~\ref{semiflow-sec} we define branched surfaces $(X,\mathfrak C,\psi)$, where $X$ is $2$-complex with a
semiflow $\psi$, and  cellular
structure $\mathfrak C$ satisfying compatibility conditions with respect to $\psi$.  To a branched surface we 
associate a dual digraph $D$ and a $G$-labeled cycle complex
$C_D$, where $G = H_1(X;\Z)/{\mbox{torsion}}$, and a cycle function $\theta_D \in \Z G$.   We show that
$\theta_D$ is invariant under certain allowable
cellular subdivisions and homotopic modifications of $(X,\mathfrak C, \psi)$.

In Sections~\ref{mappingtori_and_branchedsurfaces-sec} and \ref{foldedmappingtorus-sec} 
we study the branched surfaces associated to
the train-track map $f$ and folding sequence $\f$ defined in \cite{DKL}, called respectively the mapping torus, and folded mapping torus.
We use the invariance under allowable cellular subdivisions and modifications 
established in Section~\ref{digraphs-sec} and Section~\ref{semiflow-sec} to show that the cycle functions for these 
branched surfaces are equal.   The results of Section~\ref{digraphs-sec} applied to the mapping torus for $f$ imply the existence
of $\Theta_\phi$ and $\calT_\phi$ in Theorem~\ref{main-thm}.  Applying an argument in \cite{DKL}, we show that
 further subdivisions with the  folded mapping 
torus give rise to mapping tori for train-track maps corresponding to $\phi_\alpha$, and use this to show that
for $\alpha \in \calA_\f$,  we have $\lambda(\phi_\alpha) = |\Theta_\phi^{(\alpha)}|$.
We further compare the definition of the DKL cone $\calA_\f$ and $\calT_\phi$ to show
inclusion $\calA_\f \subset \calT_\phi$, and thus complete the proof of Theorem~\ref{main-thm}.

We conclude  in Section \ref{example-sec}  with  an example where $A_\f$ is a 
proper subcone of $\calT_\phi$.

\subsection*{Acknowledgements.} The authors would like to thank J. Birman,  
S. Dowdall, N. Dunfield,  A. Hadari , C. Leininger, C. McMullen, and K. Vogtman 
for helpful discussions.

\section{Digraphs, their cycle complexes and eigenvalues of $G$-matrices}
\label{digraphs-sec}

This section contains basic definitions and properties of digraphs, and a key result 
of McMullen that will be useful in our proof of Theorem~\ref{main-thm}.

\subsection{Digraphs, cycle complexes and their cycle polynomials}
\label{plain-digraph-sec}
We recall basic results concerning digraphs (see, for example,  \cite{Gantmacher59} 
and \cite{C-R:Graphs} for more details).
 
\begin{definition}\label{digraph-def} A {\it digraph} $D$ is a finite directed graph with at least two vertices.  Given an ordering
$v_1,\dots,v_m$ of the vertices of $D$, the  {\it adjacency matrix}
of $D$ is the matrix 
$$
M_D = [a_{i,j}],
$$
where $a_{i,j} = m$ if there are $m$ directed edges from $v_i$ to $v_j$.
The {\it characteristic polynomial} $P_D$ is the characteristic polynomial of $M_D$
and the {\it dilatation} $\lambda(D)$ of $D$ is the spectral radius of $M_D$
$$
\lambda(D) = \max \Big\{|e| \, \Big| \, \mbox{$e$ is an eigenvalue of $M_D$} \Big\}.
$$  
\end{definition}

Conversely, any square $m\times m$ matrix $M = [a_{i,j}]$ with non-negative integer entries determines 
a digraph $D$ with $M_D = M$.  The digraph $D$ has $m$ vertices and $a_{i,j}$ vertices from the $i$th to the $j$th vertex.

\begin{definition}\label{EXdigraph-def}
For a matrix $M$, let $a_{ij}^m$ be the $ij^{\, \text{th}}$ entry of $M^m$.
A non-negative  matrix $M$ with real entries is called  \emph{expanding} if 
\[
\limsup_{m \to \infty} a_{ij}^m = \infty.
\]
A digraph $D$ is \emph{expanding} if its directed adjacency matrix $M_D$ is 
expanding. Note that an expanding digraph is strongly connected. 
\end{definition}

An eigenvalue of $M$ is {\it simple} if its algebraic multiplicity is 1. 
Note that several simple eigenvalues may have the same norm. 
The following theorem is well-known (see, for example, \cite{Gantmacher59}).

\begin{theorem} \label{EX-thm} Let $M$ be a matrix and $\lambda(M)$ the maximum of the norm 
of any eigenvalue of $M$. 
If $M$ is expanding, then
it has a simple eigenvalue with norm equal to $\lam(M)$ and it has an associated   
eigenvector that is strictly positive. 
In addition, for every $i$ and $j$, we have
\[
\limsup_{m \to \infty} (a_{ij}^m)^{\frac 1m} = \lambda(M). 
\]
\end{theorem}

\begin{definition} 
A {\it simple cycle} $\al$ on a digraph $D$ is an isotopy class of embeddings
of the circle $S^1$ to $D$ oriented compatibly with the directed edges of $D$. 
A {\it cycle} is a disjoint union of simple cycles.
The  {\it cycle complex} $C_D$ of a digraph $D$ is the 
collection of cycles on $D$ thought of as a simplicial complex, whose vertices are the simple cycles.


The cycle complex $C_D$ has a measure which assigns to each cycle  its length in $D$,
that is, if $\gamma$ is a cycle on $C_D$, then its {\it length} $\ell(\gamma)$ is the number of
vertices (or equivalently the number of edges)
of $D$ on $\gamma$, and, if $\sigma = \{\gamma_1,\dots,\gamma_s\}$, then
$$
\ell(\sigma) = \sum_{i=1}^s \ell(\gamma_i).
$$
Let $|\sigma| = s$ be the {\it size} of $\sigma$.
The {\it cycle polynomial} of a digraph $D$ is given by
$$
\theta_D (x) = 1 + \sum_{\sig \in C_D} (-1)^{|\sigma|} x^{-\ell(\sigma)}.
$$
\end{definition}

\begin{theorem}[Coefficient Theorem for Digraphs \cite{C-R:Graphs}]\label{coeff-thm} 
Let $D$ be a digraph with $m$ vertices, and $P_D$ the characteristic polynomial of the 
directed adjacency
matrix $M_D$ for $D$. Then
\[
P_D(x) 
~ = ~ x^m \theta_{D}(x)
\]
\end{theorem}

\begin{proof} Let $\mbox{M}_D= [ a_{i,j} ]$  be the adjacency matrix for $D$.
Then
\begin{eqnarray*}
P_D (x) &=& \mbox{det}(xI - \mbox{M}_D).
\end{eqnarray*}
Let $S_{\sV}$ be the group of permutations of the vertices $\sV$ of
$D$.   For $\pi \in S_V$, let $\mbox{fix}(\pi) \subseteq {\sV}$ be the set of
vertices fixed by  $\pi$, and let $\sign(\pi)$ be  -1 if $\pi$ is an odd permutation  and $1$ if $\pi$ is even.
Then
\begin{eqnarray*}
P_D(x) &=& \sum_{\pi\in S_{\sV}} \sign(\pi) A_\pi
\end{eqnarray*}
where
\begin{eqnarray}\label{Ctau}
A_\pi&=& \prod_{v \notin \mbox{fix}(\pi)}(-a_{v,\pi(v)})
\prod_{v \in \mbox{fix}(\pi)} (x-a_{v,v})
\end{eqnarray}

There is a natural map $\Sigma: C_D \to S_{\sV}$ from the cycle complex $C_D$ to the permutation group
$S_{\sV}$  on the set $\sV$ defined as follows.
For each simple cycle $\gamma$ in $D$ passing through the vertices $\sV_\gamma \subset \sV$, there is a corresponding 
cyclic permutation $\Sigma(\gamma)$ of $\sV_\gamma$.  That is, if $\sV_\gamma = \{v_1,\dots,v_\ell\}$ contains more than one vertex
and is ordered according to their appearance in the cycle, then
then $\Sigma(\gamma) (v_i) = v_{i+1 (\mbox{\small mod} \ell)}$.   If $\sV_\gamma$ contains one vertex, we 
say $\gamma$ is a {\it self-edge}.  For self-edges $\gamma$,
$\Sigma(\gamma)$ is the identity permutation.
Let $\sigma = \{\gamma_1,\dots,\gamma_s\}$ be a cycle on $D$.  Then we define $\Sigma(\sigma)$
to be the product of disjoint cycles
$$
\Sigma(\sigma) = \Sigma(\gamma_1) \circ \cdots \circ \Sigma(\gamma_\ell).
$$

The polynomial $A_\pi$ in Equation~(\ref{Ctau}) can be rewritten in terms of the cycles $\sigma$ of $C_D$ with $\Sigma(\sig) = \pi$. 
First we rewrite $A_\pi$ as
\begin{eqnarray}\label{summand-eqn}
A_\pi = \sum_{\nu \subset \mbox{fix}(\pi)}x^{|\mbox{fix}(\pi) - \nu|} \prod_{v \notin \mbox{fix}(\pi)}(-a_{v,\pi(v)})\prod_{v \in \nu} (- a_{v,v} ).
\end{eqnarray}

Let  $\pi \in S_{\sV}$ be in the image of $\Sigma$.  For a cycle $\sigma \in C_D$, let  $\nu(\sigma) \subset \sV$
be the subset vertices at which $\sigma$ has a self-edge.

For $\nu \subset \mbox{fix}(\pi)$ let 
$$
P_{\pi,\nu} = \{ \sig \in C_D\ |\ \mbox{$\Sigma(\sig) = \pi$ and $\nu(\sig)= \nu$}\}.
$$
The we claim that the number of elements in $P_{\pi_\nu}$ is 
\begin{eqnarray}\label{Ppinu-eqn}
\prod_{v \notin \mbox{fix}(\pi)} a_{v,\pi(v)} \prod_{v \in \nu} a_{v,v}.
\end{eqnarray}
Let $\sigma \in C_D$ be such that $\Sigma(\sigma) = \pi$.  Then for each $v \in \sV \setminus \mbox{fix}(\pi)$,
there is a choice of $a_{v,\pi(v)}$ edges from $v$ to $\pi(v)$, and for each $v \in \mbox{fix}(\pi)$ 
$\sigma$ either contains no self-edge, or one of $a_{v,v}$ possible self-edges
at $v$.  This proves (\ref{Ppinu-eqn}).

For each $\sigma \in C_D$, we have
$$
\ell(\sigma) = m - |\mbox{fix}(\Sigma(\sigma))| + |\nu(\sigma)|.
$$
Thus, the summand in (\ref{summand-eqn}) associated to $\pi \in S_{\sV} \setminus \mbox{id}$, and $\nu \subset \mbox{fix}(\pi)$ is given
by
\begin{eqnarray*}
x^{|\mbox{fix}(\pi)|-|\nu|}\prod_{v \notin \mbox{fix}(\pi)} (-a_{v,\pi(v)}) \prod_{v \in \nu} (- a_{v,v}) 
&=&
(-1)^{m- |\mbox{fix}(\pi)| + |\nu|}\sum_{\sigma \in P_{\pi,\nu}} x^{m - \ell(\sigma)}\\
&=&
\sum_{\sigma \in P_{\pi,\nu}} (-1)^{\ell(\sigma)} x^{m - \ell(\sigma)},
\end{eqnarray*}
and similarly for  $\pi = \mbox{id}$ we have
\begin{eqnarray*}
A_\pi &=& \prod_{v \in \sV} (x - a_{v,v})\\
&=&
x^m + \sum_{\sigma \in P_{\pi,\nu}} (-1)^{\ell(\sigma)} x^{m - \ell(\sigma)}.
\end{eqnarray*}
For each $\sigma \in C_D$, $\sign (\Sigma(\sigma)) = (-1)^{\ell(\sigma) - |\sigma|}$.
Putting this together, we have
\begin{align*}
P_D(x) &= \sum_{\pi \in S_\sV} \sign(\pi) A_\pi\\
&= x^m +  \sum_{\pi \in S_\sV} \sum_{\sigma \in C_D\from\ \Sigma(\sigma) = \pi}
(-1)^{\ell(\sigma) - |\sigma|} (-1)^{\ell(\sigma)} x^{m-\ell(\sigma)}\\
&= x^m + \sum_{\sigma \in C_D} (-1)^{|\sigma|} x^{m- \ell(\sigma)}. \qedhere
\end{align*}
\end{proof}

\subsection{McMullen Cones}\label{McMullen-cone-sec}  Each group ring element partitions $\Hom(G;\R)$ into a union of
cones defined below.

\begin{definition}(cf. \cite{McMullen_Alexander_norm})\label{McMullen_cone-def}
Let $G$ be a finitely generated free abelian group.
Given an element $\theta =  \sum_{g \in G} a_g g \in \ZZ G$, 
 the {\it support} of $\theta$ is the set
$$
\Supp(\theta) = \{g \in G  \ | \ a_g \neq 0\}.
$$
Let $\theta \in \ZZ G$ and $g_0 \in \Supp(\theta)$ the {\it McMullen cone} of $\theta$ for $g_0$ is the set 
$$
\calT_\theta(g_0) = \{\alpha \in \Hom(G;\R) \mid  \alpha(g_0) > \alpha (g)\  \mbox{for all $g \in \Supp(\theta) \setminus \{g_0\}$}\}.
$$
\end{definition}

\begin{remark}\label{NewtonPolygon-rem}
The elements of $G$ can be identified with a subset of the dual
space 
$$
\widehat{ \Hom(G;\R)} = \Hom(\Hom(G;\R),\R)
$$
 to $\Hom(G;\R)$.  Let $\theta \in \Z G$ be any element.
The convex hull of $\Supp(\theta)$ in $\widehat{ \Hom(G;\R)}$ is called the {\it Newton polyhedron} $\mathcal N$ of $\theta$.
Let $\widehat{ \mathcal N}$ be the dual of $\mathcal N$ in $\Hom(G;\R)$.  That is, each top-dimensional
face of $\widehat {\mathcal N}$
corresponds to a vertex $g \in \mathcal N$, and each $\alpha$ in the cone over this face
has the property that $\alpha (g) > \alpha (g')$ where $g'$ is any vertex of $\mathcal N$ with $g\neq g'$.  
Thus, the McMullen cones $\calT_\theta(g_0)$, for $g_0 \in \Supp(\theta)$ are the cones over the top dimensional faces of
the dual  to the Newton polyhedron of $\theta$.
\end{remark}

\subsection{A coefficient theorem for $H$-labeled digraphs}\label{labeled-digraphs-sec}
Throughout this section let $H$ be the free abelian group with $k$ generators and let 
$\Z H$ be its group ring. Let $G = H \times \langle s \rangle$, where $s$ is an extra 
free variable.   Then the Laurent polynomial ring $\Z H (u)$ is canonically isomorphic to 
$\Z G$, by an isomorphism that sends $s$ to $u$.

We generalize the results of Section~\ref{plain-digraph-sec} to the setting of 
$H$--labeled digraphs.  

\begin{definition}\label{cyclelabeling-def}
Let $C$ be a simplicial complex. An \emph{$H$--labeling} of $C$ is a map 
$$
h \from C \rightarrow H
$$
compatible with the simplicial complex structure of $H$, i.e., 
$$
h(\sigma) = \sum_{i=1}^\ell h(v_i)
$$
for  $\sigma = \{v_1,\dots,v_\ell\}$.
An { \it $H$-complex}\ $\calC^H$ is an abstract simplicial complex together with a $H$-labeling. 
\end{definition}

\begin{definition}\label{CycleFunction-def} The {\it cycle function} of an $H$--labeled 
complex $\calC^H$ is the element of $\Z H$ defined by
$$
\theta_{\calC^H}  = 1 + \sum_{\sigma \in \calC^H} (-1)^{|\sigma|} h(\sigma)^{-1}.
$$
\end{definition}

\begin{definition}\label{Glabeled-def} An {\it $H$-digraph} $\Dlab^H$ is a  digraph $D$, along with a map 
$$
h \from \calE_D \to H,
$$
where $\calE_D$ is the set of edge of $D$. 
 The digraph $D$ is the {\it underlying digraph}  of $\Dlab^H$.

An $H$-labeling on a digraph induces an $H$-labeling on its cycle complex.
Let $\gamma$ be a simple cycle on $D$.  Then up to isotopy, $\gamma$ can
be written as
$$
\gamma = e_0 \cdots e_{k-1},
$$
for some collection of edge $e_0,\dots,e_{k-1}$ cyclically joined end to end on $D$.
Let 
$$
h(\gamma) = h(e_0) + \ldots + h(e_{k-1}),
$$
and for $\sigma = \{\gamma_1,\dots,\gamma_\ell\}$, let
$$
h(\sigma) = \sum_{i=1}^{\ell} h(\gamma_i).
$$
Denote  the labeled cycle complex by $\calC_\Dlab^H$.The {\it cycle polynomial} 
$\theta_{\Dlab^H}$ of $\Dlab^H$ is given by
$$
\theta_{\Dlab^H} (u)= 1 + \sum_{\sigma \in \C_\Dlab^H} (-1)^{|\sigma|} h(\sigma)^{-1} u^{-\ell(\sigma)} \in \Z H [u] = \Z G.
$$
\end{definition}

The cycle polynomial of $\theta_{\Dlab^H}(u)$ contains both the information about the
associated labeled complex $\calC_\Dlab^H$ and the length functions on cycles on $D$.
One observes the following by comparing Definition~\ref{CycleFunction-def} and Definition~\ref{Glabeled-def}.

\begin{lemma}\label{cyclepoly_v_cyclefunction}
The cycle polynomial of the $H$--labeled digraph $\Dlab^H$, and the cycle function 
of the labeled cycle complex $\calC_\Dlab^H$ are related by
$$
\theta_{\calC_\Dlab^H} = \theta_{\Dlab^H}(1).
$$
\end{lemma}

\begin{definition}\label{pos-def} 
An element $\theta \in \ZZ H$ is \emph{positive}, denoted $\theta>0$, 
if 
$$
\theta = \sum_{h \in H} a_h h,
$$
where $a_h \ge 0$ for all $h \in H$, and $a_h > 0$ for at least one $h \in H$.
If $\theta$ is positive or $0$ we say that it is non-negative and write $\theta \geq 0$. 
\end{definition}

A matrix $M^H$ with entries in $\ZZ H$ is called an \emph{$H$--matrix.}  If all entries 
are non-negative, we write $M^H \ge 0$ and if all entries are positive we write 
$M^H > 0$.
 
\begin{lemma} 
There is a bijective correspondence between $H$--digraphs $\Dlab^H$ and 
non-negative $H$--matrices $M_\Dlab^H$, so that $M_\Dlab^H$ is the directed 
incidence matrix for $\Dlab^H$.
\end{lemma}

\begin{proof}  
Given a labeled digraph $\Dlab^H$, let $E_{ij}$ be the set of edges from the 
$i^{\, \text{th}}$ vertex to the $j^{\, \text{th}}$ vertex. We form a matrix $M_\Dlab^H$ 
with entries in $\ZZ H$ by setting 
\[ a_{ij} = \sum_{e \in E_{ij}} h(e), \]
where $h(e)$ is the $H$-label of the edge $e$.

Conversely, given an
 $n \times n$  matrix $M^H$ with entries in $\ZZ H$,
let $\Dlab^H$ be the $H$-digraph with $n$ vertices $v_1,\dots,v_n$ and, for each $i,j$  with
$m_{i,j}=\sum_{h \in H} a_g g \geq 0$, it has $a_h$ directed edges from $v_i$ to $v_j$ labeled by $h$.  The directed incidence matrix $M^H_{\Dlab}$
equals $M$ as desired.
\end{proof}

The proof of the next theorem is similar to that of the Theorem~\ref{coeff-thm} and is 
left to the reader. 

\begin{theorem}[The Coefficients Theorem for $H$--labeled digraphs]
\label{CTlabeled-thm}
Let $\Dlab^H$ be an $H$--labeled digraph with $m$ vertices, and 
$P_\Dlab(u)\in \Z H[u]$ be the characteristic polynomial of its incidence matrix. 
Then,
$$
P_\Dlab(u) =  u^m \theta_{\Dlab^H}(u).
$$
\end{theorem}

Note that $P_\calD(u)$ is an element of $\Z G$. 

\subsection{Expanding $H$--matrices}\label{Gmat-sec}
In this section we recall a key theorem of McMullen on leading eigenvalues of 
specializations of expanding $H$--matrices (see \cite{McMullen_polynomial_invariants}, 
Appendix A). Although McMullen's theorem is stated for Perron-Frobenius
matrices, the proof only uses only that the matrix is expanding. 

\begin{definition}\label{PF_Gmat-def}
A labeled digraph $\calD^H$ is called \emph{expanding} if the underlying 
digraph $\calD$ is expanding. The $H$--matrix $M_\calD^H$ is defined to be 
\emph{expanding} if the associated labeled digraph $\calD^H$ is expanding. 
\end{definition} 

For the rest of this section, we fix an expanding $H$--labeled digraph $\calD^H$. 
Consider an element $\t \in \Hom(H,\RR_+)$. Define $M_\calD^H(\t)$ to be the real valued matrix 
obtained by applying $\t$ to the entries of $M_\calD^H$ (where $\t$ is extended linearly
to $\Z H$). Alternatively,  identify $H$ with space of monomials in 
$k$ variables $t_1, \ldots, t_k$. This gives a natural identification 
of $\Hom(H,\RR_+)$ with $\R_+^k$ where the $i^{\, \text{th}}$ coordinate in
$\R_+^k$ is associated to the variable $t_i$. Then $M_\calD^H(\t)$ is
be the matrix obtained by replacing $t_i$ with $i^{\, \text{th}}$ coordinate of 
$\t \in \R_+^k= \Hom(H,\RR_+)$. 

Note that, since $\calD^H$ is expanding, for every $\t \in \R_+^k$, the real valued 
matrix $M_\calD^H(\t)$ is also expanding. Define a function
\[
E \from \R_+^k \to \R_+, 
\qquad\text{by}\qquad  
E(\t) = \lambda\big(M_\calD^H(\t) \big).
\]
Identifying the ring $\Hom(H,\RR)$ with $\RR^k$, there is a 
natural map 
\[
\exp \from \Hom(H,\RR) \to \Hom(H,\RR_+),
\]
where, for $\w = (w_1, \ldots, w_k) \in \R^k$, 
\[
\exp(\w) = (e^{w_1}, \ldots, e^w_k).
\]  
Define
\[
\delta \from \R^k \to \R, \qquad\text{by}\quad \delta(\w)= \log E(e^\w).  
\] 
Note that the graph of the function $\delta$ lives in $\R^k \times \R$
which can be naturally identified with $\Hom(G; \R)$. 
 
\begin{theorem}[McMullen \cite{McMullen_polynomial_invariants}, Theorem A.1]
\label{McMullen-thm} 
For an expanding $H$--labeled digraph $\calD^H$, we have the following.
\begin{enumerate}
\item The function $\delta$ is real analytic and convex.
\item The graph of $\delta$ meets every ray through the origin of 
$\R^k \times \R$ at most once.
\item\label{thisitem} 
For $Q(u)$ any factor of $P_\calD(u)$, where $Q(E(\t)) = 0$ for all 
$\t \in \R_+^k$, and for $d = \deg(Q)$, the set of rays passing 
through the graph of $\delta$ in $\R^k \times \R$ coincides with the 
McMullen cone $\calT_Q(u^d)$.
\end{enumerate}
\end{theorem}


\begin{definition} 
For any expanding $H$--labeled digraph $\calD^H$, let
$d=\deg(P_\calD)$. We refer to the cone $\calT = \calT_{P_\calD}(u^d)$
as the \emph{McMullen cone} for the element $P_\calD \in \Z G$. 
Alternatively we refer to it as the McMullen cone for the $H$--matrix 
$M_\calD^H$.
\end{definition}

\begin{theorem}[McMullen \cite{McMullen_polynomial_invariants}]\label{McM-thm} 
For any expanding $H$--labeled digraph $\calD^H$ the map 
\[
L \from \Hom(G; \Z) \rightarrow \R
\] 
defined by
\[
L(\alpha) =  \log |P_\calD^{(\alpha)}|,
\]
extends to a homogeneous of degree $-1$, real analytic, convex function on 
the McMullen cone $\calT$ for the element $P_\calD$. It 
goes to infinity toward the boundary of affine planar sections of $\calT$.
\end{theorem}

Theorem~\ref{McM-thm} summarizes results taken from \cite{McMullen_polynomial_invariants} given in the context
of mapping classes on surfaces. For the convenience of the reader, we give a proof here.

\begin{proof} The function $L$ is real analytic since the house of a polynomial is 
an algebraic function in its coefficients.  
Homogeneity of $L(z)$ follows from the following observation: $\rho$ is a root of 
$Q(x^{\w}, x^{s})$ if and only if $\rho^{1/c}$ is a root of $Q(x^{c\w},x^{cs})$. Thus
\[ L(cz) = \log|Q(x^{c\w}, x^{cs})| = c^{-1} \log |Q(x^\w, x^s)| = c^{-1} L(z).
\]

By homogeneity of $L$, the values of $L$ are determined by the values at any level 
set, one of which is the graph of $\delta(\w)$. To prove convexity of $L$, we 
show that level sets of $L$ are convex i.e. the line connecting two points on a level 
set lies above the level set. Let $\Gamma = \{ z=(\w,s) \mid L(z) = 1\} $ and  
$\Gamma' = \{ z=(\w,s) \mid  s = \delta(\w)\}$. We show that  $\Gamma = \Gamma'$. 
It then follows that, since $\Gamma'$ is a graph of a convex function by Theorem 
\ref{McMullen-thm}, $\Gamma$ is convex.  

We begin by showing that $\Gamma' \subset \Gamma$ (cf. \cite{McMullen_polynomial_invariants}, proof Theorem 5.3).
 If $\beta = (\a,b) \in \Gamma'$ then 
$\delta(\a)=b$, hence $Q(e^\a,e^b)=0$ and $|Q(e^\a, e^b )| \geq  e$. Let  $r = L(\beta) = \log|Q(e^\a, e^b )|$.
Since $b=\delta(\a)$, by the convexity of the function $\delta$, we have
$rb \geq \delta(r\a)$. On the other hand, $Q(e^{r\a},e^{rb})=0$ hence $e^{rb}$ 
is an eigenvalue of $M(e^{r\a})$ so $rb \leq \log E(e^{r\a}) = \delta(r\a)$. We get that $rb = \delta(r\a)$. The points 
$(\a,b), (r\a,rb)$ both lie on the same line through the origin so by 
Theorem \ref{McMullen-thm} part (2), they are equal.
Thus $r=1=L(\beta)$, and hence $\beta \in \Gamma$. 

To show that $\Gamma \subset \Gamma'$ in $\calT$, note that every ray in $\calT$ initiating from the origin intersects $\Gamma$ because it intersects $\Gamma'$ by part (3) of Theorem \ref{McMullen-thm}. Because $L$ is homogeneous, level sets of $L$ intersect every ray from the origin at most once. Therefore, in $\calT$, $\Gamma=\Gamma'$ and the latter is the graph of a convex function.

We now show that if $L$ is a homogeneous function of degree -1, and has convex level sets then $L$ is convex
(cf.  \cite{McMullen_polynomial_invariants} Corollary 5.4). 
This is equivalent to showing that $1/L(z)$ is concave on $\calT$.
Let $z_1,z_2 \in T$ lie on distinct rays through the origin, and let 
$$
z_3 = s z_1 + (1-s) z_2.
$$
Let $c_i$, $i=1,2,3$, be constants so that $z_i' = c_i^{-1}z_i$ is in the level set $L(c_i^{-1}z_i) = 1$.  
Let $p$ lie on the line $[z_1',z_2']$ and on the ray through $z_3$. Then $p$ has the form 
$$
p = r z_1' + (1-r) z_2' 
$$
for $0<r<1$. If
$$
r=\frac{sc_1}{s c_1+(1-s) c_2}
$$
then we have
$$
p = \frac{z_3}{sc_1' + (1-s)c_2'}.
$$
 Since the level set for $L(z) = 1$ is convex, $p$ is equal to or above $z_3/c_3$, and we have 
\begin{eqnarray}\label{convex-eqn}
1/(sc_1 + (1-s)c_2) \ge 1/c_3.
\end{eqnarray}
Thus
\begin{eqnarray}\label{convex2-eqn}
1/L(z_3) =  c_3 \ge s c_1 + (1-s) c_2 = s/L(z_2) + (1-s)/L(z_3).
\end{eqnarray}
Thus $1/L(z)$ is concave, and hence $L(z)$ is convex. 

Let $z_n$ be a sequence of points on an affine planar section of $\calT$ approaching the boundary
of $\calT$.  Let $c_n$ be such that $c_n^{-1}z_n$ is in the level set $L(z) = 1$.  Then $L(z_n) = c_n^{-1}$ 
for all $n$.  But $z_n$ is bounded, while the level set $L(z) = 1$ is asymptotic to the boundary of $\calT$.
Therefore, $1/L(z_n)$ goes to 0 as $n$ goes to infinity.
  \end{proof}

\begin{remark} If the level set $L(z) = 1$ is strictly convex, then $L(z)$ is strictly convex. Indeed, if  $L(z) = 1$ is strictly convex, then the inequality in (\ref{convex-eqn}) is strict,
and hence the same holds for (\ref{convex2-eqn}).
\end{remark}

\subsection{Distinguished factor of the characteristic polynomial}
We define a distinguished factor of the characteristic polynomial of a Perron-Frobenius $G$-matrix.

\begin{proposition}\label{distinguishedFactor}  Let $P$ be the characteristic polynomial of a Perron-Frobenius $G$-matrix.
Then $P$ has a factor $Q$ with the properties:
\begin{enumerate}
\item  for all integral elements $\alpha$ in the McMullen cone  $\calT$,
$$
 |P^{(\alpha)}| = |Q^{(\alpha)}|,
$$
\item {\it minimality}: if  $Q_1 \in \Z G[u]$ satisfies $|Q^{(\alpha)}| = |Q_1^{(\alpha)}|$ for
all $\alpha$ ranging among the integer points of an open subcone of $\calT$,
then $Q$ divides $Q_1$,  and
\item if $r$ is the degree of $Q$ the cones $\calT_P(u^d)$ and $\calT_Q(u^r)$ are equal, where $d$ is the degree of $P$ and 
$r$ is the degree of $Q$
as elements of $\Z G[u]$.
\end{enumerate}
\end{proposition}

\begin{definition} Given a Perron-Frobenius $G$-matrix $M^G$, the polynomial $Q$ is called the {\it distinguished factor} of the characteristic polynomial 
of $M^G$.
\end{definition}

\begin{lemma}\label{minpoly-lem} 
Let $F(\t) \from \RR^k \to \RR$  be a function, consider the ideal
\[ I_{F} = \{ \theta \in \ZZ(\t)[u] \mid \theta(\t,F(\t))=0 \text{ for all } \t \in \RR^k \} \] 
$I_F$ is a principal ideal. 
\end{lemma}

\begin{proof}
Let $\overline{\ZZ(\t)}[u]$ be the ring of polynomials in the variable $u$ over the quotient field of $\ZZ(\t)$.  
Since  $\overline{\ZZ}(\t)[u]$ is a principal ideal domain, $I_F$ generates a principal ideal
$\overline I_F$ in $\overline{\ZZ(\t)}[u]$.  

Let $\overline \theta_1$ be a generator of $\overline I_F$, then $\overline{\theta_1} = \frac{\theta_1(\t,u)}{\sig(\t)}$ with $\theta_1 \in I_F$. Thus $\overline\theta_1(\t,F(\t))=0$ for all $\t$. 
If $I_F$ is the zero ideal then there is nothing to prove, therefore we suppose it is not. Let 
$
\overline \theta_1(\t,u) = \frac{\nu(\t,u)}{\delta(\t)},
$ 
where $\nu$ and $\delta$ are relatively prime in $\overline{\ZZ(\t)}[u]$, a unique factorization domain.  
Since $\theta_1(\t,F(\t))=0$ for all $\t$ then $\nu(\t,F(\t))=0$ for all $\t$ and $\nu \in I_F$. 

Since $ I_F$ is not the zero ideal then $\overline I_F$ is not the zero ideal, hence $\bar \theta_1 \neq 0$ which 
implies that $\nu \neq 0$. Let $\theta \in I_F$ be any polynomial. Since $\overline\theta_1$ divides $\theta$, then 
$\nu$ divides $\theta\delta(\t)$ but since $\nu$ and $\delta$ are relatively prime, $\nu$ divides $\theta$. 
Therefore, $\nu$ is a generator of $I_E$.
\end{proof}

\begin{proof}[Proof of Proposition \ref{distinguishedFactor}] The proposition follows from Lemma \ref{minpoly-lem} by declaring $Q$ to be the generator of $I_L$ for $L \from \calT \to \RR$ defined in Theorem \ref{McM-thm}. 
\end{proof}

\section{Branched surfaces with semiflows}\label{semiflow-sec}

In this section we associate a digraph and an element $\theta_{X,\mathfrak C,\psi} \in \Z G$ to a branched surface $(X,\mathfrak C, \psi)$.  
We show that this element is invariant under certain kinds of  subdivisions of $\mathfrak C$.

\subsection{The cycle polynomial of a branched surface with a semiflow}\label{dualdigraph-sec}

\begin{definition}\label{branchsurface-def}
Given a 2-dimensional CW-complex $X$, 
a {\it semiflow}  on $X$ is a continuous map   $\psi\from X \times \RR_+ \to X$ satisfying
\begin{enumerate}[(i)]
\item $\psi(\cdot, 0) \from X \to X$ is the identity, 
\item $\psi( \cdot , t) \from X \to X$ is a homotopy equivalence for every $t \geq 0$, and
\item $\psi( \psi(x,t_0), t_1) = \psi(x,t_0+t_1)$ for all $t_0,t_1 \geq 0$. 
\end{enumerate} 
A cell-decomposition $\mathfrak C$ of $X$ is  {\it $\psi$-compatible} if the following hold.
\begin{enumerate}
\item Each $1$--cell is
either  contained in a flow line ({\it vertical}), or transversal to the semiflow
at every point ({\it transversal}).
\item\label{forwardflow-item} For every vertex $p \in \mathfrak C^{(0)}$, the image of the {\it forward flow} of $p$,
$$
\{\psi(p,t) \ | \ t \in \R_{>0}\},
$$
 is contained in $\mathfrak C^{(1)}$.
\end{enumerate}
A {\it branched surface} is a triple $(X,\mathfrak C,\psi)$, where $X$ is a 2-complex with semi-flow $\psi$ and a
$\psi$-compatible cellular structure $\mathfrak C$.
\end{definition}

\begin{remark} We think of branched surfaces as flowing downwards.  From this point of view,
Property~(\ref{forwardflow-item}) implies that every $2$--cell $c \in \mathfrak C^{(2)}$ has a
unique {\it top $1$--cell}, that is, a $1$--cell $e$ such that each point in $c$ can be realized as the forward orbit of
a point on $e$.  
\end{remark}
 
 \begin{definition}
Let $e$ be a $1$--cell on a branched surface $(X,\mathfrak C, \psi)$ that is transverse to the flow at every point.  
A \emph{hinge} containing $e$  is an equivalence class of
 homeomorphisms $\kappa \from [0,1] \times [-1,1] \into X$ so that:
\begin{enumerate}
\item  the half segment $\Delta = \{ (x, 0)\mid x \in I \}$ is mapped onto $e$,
\item the image of the interior of the $\Delta$ intersects $\mathfrak C^{(1)}$ only in $e$, and 
\item  the vertical line segments $\{x\}\times [-1,1]$ are mapped into flow lines on $X$.
\end{enumerate} 
Two hinges $\kappa_1,\kappa_2$ are {\it equivalent} if there is an isotopy rel $\Delta$ between them. 
The $2$--cell on $(X,\mathfrak C,\psi)$ containing $\kappa([0,1] \times [0,1])$ is called the {\it initial cell} of $\kappa$ and the
$2$--cell containing the point $\kappa([0,1] \times [-1,0])$ is called the {\it terminal cell} of $\kappa$.
\end{definition}

An example of a  hinge is illustrated in Figure~\ref{fig:hinge}.

\begin{figure}[htbp] 
   \centering
   \includegraphics[height=1.2in]{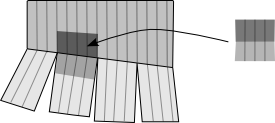} 
   \caption{A hinge on a branched surface.}
   \label{fig:hinge}
\end{figure}

\begin{definition}\label{dualdigraph-def} Let $(X,\mathfrak C, \psi)$ be a  branched surface.  The {\it dual digraph} $D$ of $(X,\mathfrak C,\psi)$ is the digraph with  a vertex for every $2$--cell and
an edge for every hinge $\kappa$ from the vertex corresponding to its initial $2$--cell  to the vertex corresponding to
its terminal $2$--cell. The dual digraph $D$ for $(X,\mathfrak C,\psi)$ embeds into $X$
$$
D \hookrightarrow X
$$
so that each vertex is mapped into the interior of the corresponding
$2$--cell, and each directed edge is mapped into the union of the two-cells corresponding
to its initial and end vertices,  and intersects the common boundary of the $2$--cells at a single point. 
The embedding is well defined up to homotopies of $X$ to itself.
\end{definition}

An example of an embedded dual digraph is shown in Figure~\ref{fig:dualdigraph}.  In this example, there are three edges
emanating from $v$ with endpoints at $w_1, w_2$ and $w_3$.  It is possible that $w_i = w_j$ for some $i \neq j$, 
of that $w_i = v$ for some $i$.  These cases can be visualized using Figure~\ref{fig:dualdigraph}, where we identify the
corresponding $2$--cells.

\begin{figure}[htbp] 
   \centering
   \includegraphics[height=1.2in]{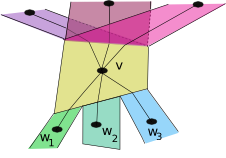} 
   \caption{A section of an embedded dual digraph.}
   \label{fig:dualdigraph}
\end{figure}

Let $G = H_1(X;\Z)/{\mbox{torsion}}$, thought of as the integer lattice  in $H_1(X,\RR)$.  The embedding of $D$ in $X$ determines 
a $G$-labeled cycle complex $\calC_D^G$ where
for each $\sig \in \calC_D^G$ 
 and $g(\sig)$ is the homology class of the cycle $\sigma$ considered as a 1-cycle on $X$. 
 
\begin{definition}\label{cyclefunction-def}
Given a branched surface $(X, \mathfrak C, \psi)$, the {\it cycle function} of $(X,\mathfrak C,\psi)$ is the group ring element
$$
\theta_{X,\mathfrak C,\psi} = 1+ \sum_{\sig \in \calC_D^G} (-1)^{|\sig|} g(\sig)^{-1}  \in \Z G.
$$
Then we have
$$
\theta_{X,\mathfrak C,\psi}  = \theta_{\calC_D^G}(1)
$$ 
where $\theta_{\calC_D^G}(u)$ is the cycle polynomial of $\calC_D^G$.
\end{definition}

\subsection{Subdivision}\label{vertical-sec}

We show that the cycle function of $(X,\mathfrak C,\psi)$ is not invariant under certain kinds of cellular subdivisions.

\begin{definition}\label{vertsub-def} Let $p \in \mathfrak C^{(1)}$ be a point in the interior of a transversal edge in
$\mathfrak C^{(1)}$.
 Let $x_0 = p$ and inductively define $x_i = \psi(x_{i-1}, s_i)$, for $i=1,\dots,r$, so that 
$$
s_i = \min\{s \ |\ \mbox{$\psi(x_{i-1},s)$ has endpoint in $\mathfrak C^{(1)}$}\}.
$$
The {\it vertical subdivision of $X$ along the forward orbit of $p$} is
the cellular subdivision $\mathfrak C'$ of $\mathfrak C$ obtained by adding the edges $\psi(x_{i-1},[0,s_i])$, for $i=1,\dots,r$,
and subdividing the corresponding $2$--cells.
 If $x_r$ is a vertex in the original skeleton $\mathfrak C^{(0)}$ of $X$,
then we say the vertical subdivision is {\it allowable}.
 \end{definition}

\begin{figure}[htbp] 
   \centering
   \includegraphics[height=1.2in]{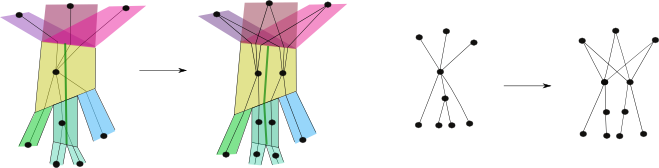} 
   \caption{An allowable vertical subdivision, and effect on the directed dual digraph.}
   \label{fig:vertsub}
\end{figure}

Figure 1 illustrates an allowable vertical subdivision with $r=2$.

\begin{proposition}\label{verticalsubdivision-prop}  Let $(X,\mathfrak C', \psi)$ be obtained from $(X,\mathfrak C,\psi)$
by allowable vertical subdivision.  Then the cycle function
$\theta_{X,\mathfrak C,\psi}$ and $\theta_{X,\mathfrak C',\psi}$ are equal.
\end{proposition}

We establish a few lemmas before proving Proposition~\ref{verticalsubdivision-prop}.  

\begin{lemma}\label{quotient-lem} 
Let $(X,\mathfrak C', \psi)$ be obtained from $(X,\mathfrak C,\psi)$
by allowable vertical subdivision.
 Let $D'$ and $D$ be the dual digraphs for $(X,\mathfrak C',\psi)$ and $(X,\mathfrak C,\psi)$.
There is a quotient map  $q\from D' \rightarrow D$ that is induced by a continuous map from $X$ to
itself that is homotopic to the identity, and 
in particular the diagram 
$$
\xymatrix{
&H_1(D';\Z)\ar[dr]\ar[r]^{q_*} &H_1(D;\Z)\ar[d]\\
&&H_1(X;\Z)
}
$$
commutes.
\end{lemma}

\begin{proof}
Working backwards from the last vertically subdivided cell to the first, each allowable vertical subdivision decomposes into a 
sequence of allowable vertical subdivisions that
involve only one $2$--cell. 
An illustration is shown in  Figure~\ref{onecellsubdivision-fig}.

\begin{figure}[htbp] 
   \centering
   \includegraphics[height=1.2in]{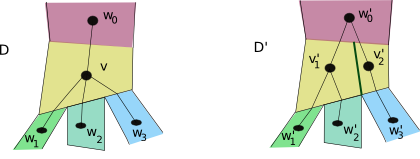} 
   \caption{Vertical subdivision of one cell.}
   \label{onecellsubdivision-fig}
\end{figure}

Let $v$ be the vertex of $D$ corresponding to the cell $c$ of $X$ that contains the new edge. 
The digraph $D'$ is constructed from $D$ by the following  steps:
\begin{enumerate}[{1.}]
\item Each vertex $u\neq v$ in $D$ lifts to a well-defined vertex $u'$ in $D'$. 
The vertex $v \in D$ lifts to two vertices $v_1', v_2'$ in $D'$. 
\item For each edge $\eps$ of $D$ neither of whose endpoints $u$ and $w$ equal $v$, the quotient map is 1-1 over $\eps$, and
hence there is only one possible lift $\epsilon'$ from
$u'$ to $w'$.
\item For each edge $\eps$ from $w \neq v$ to $v$ there are two edges $\eps'_1, \eps'_2$ where $\eps'_i$ begins
at $w'$ and ends at $v_i'$. 
\item For each outgoing edge $\eps$ from $v$ to $w$ (where $v$ and $w$ are possibly equal), there is a representative
 $\kappa$ of the hinge corresponding to $\eps$ that is contained in the union of two $2$--cells in the $\mathfrak C'$.  
 This determines a unique edge $\eps'$ on $D'$ that lifts $\eps$.
\end{enumerate}
There is a continuous map homotopic to the identity from $X$ to itself that
restricts to the identity on every cell other than $c$ or $c_w$,
where $c_w$ corresponds to a vertex $w$ with an edge from $w$ to $v$ in $D$.  On $c \cup c_w$ the map 
merges the edges $\eps'_1, \eps'_2$ so that their endpoints $v_i'$ merge to the one vertex $v$.
\end{proof}

\begin{lemma}\label{unique-lift-lem}   The quotient map $q : D \rightarrow D'$ induces an inclusion
$$
q^*: C_D \hookrightarrow C_{D'},
$$
which preserves lengths, sizes, and labels, so that for $\sigma \in C_D$, $q(q^*(\sigma)) = \sigma$.
 \end{lemma}

\begin{proof} Again we may assume that the subdivision involves a vertical subdivision of one $2$--cell $c$
corresponding to the vertex $v \in D$ and then use induction.   It is enough to define lifts  of simple cycles on $D$ to 
a simple cycle in $D'$.
All edges in $D$ from $u$ to $w$ with $w \neq v$ have a unique lift in $D'$.
Thus, if $\gamma$ does not contain $v$ then there is a unique $\gamma'$ in $D'$ such that $q(\gamma') = \gamma$.
Assume that $\gamma$ contains $v$. If $\gamma$ consists of a single edge $\eps$, then $\eps$ is
a self-edge   from $v$ to itself, and
 $\eps$ has two lifts: a self-edge from $v_1'$ to $v_1'$  and an edge from $v_1'$ to $v_2'$,
  where $v_1'$ is the vertex corresponding to the initial cell of the hinge containing $\eps$. Thus, there is a well-defined
  self-edge $\gamma'$ lifting $\gamma$ (see Figure~\ref{selfedge-fig}).
  
  \begin{figure}[htbp] 
   \centering
   \includegraphics[height=1.2in]{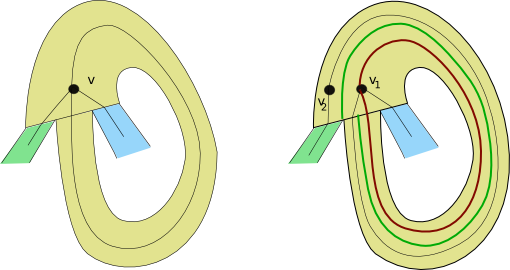} 
   \caption{Vertical subdivision when digraph has a self edge.}
   \label{selfedge-fig}
\end{figure}

Now suppose $\gamma$ is not a self-edge and contains $v$.  Let $w_1,\dots,w_{\ell-1}$ be the vertices in $\gamma$
other than $v$ in their induced sequential order.  Let $\eps_i$ be the edge from $w_{i-1}$ to $w_i$
for $i=2,\dots,\ell-1$.  Then since none of the $\eps_i$ have initial or endpoint $v$, they have unique lifts $\eps_i'$ in $D'$.
Since the vertical subdivision is allowable, there is one vertex, say $v_1'$, above $v$ with an edge $\eps_1'$ from
$ v_1'$ to $w_2'$.  Let $\eps_\ell'$ be the edge from $w_{\ell-1}'$ to $v_1'$ (cf. Figure~\ref{onecellsubdivision-fig}).  
Let $\gamma'$ be the simple cycle with edges $\eps_1',\dots,\eps_\ell'$.

Since the lift of a simple cycle is simple, the lifting map determines a well-defined map $q^*: C_D \rightarrow C_{D'}$
that  satisfies $q \circ q^* = \mbox{id}$ and preserves size.
The commutative diagram in Lemma~\ref{quotient-lem} implies that the images of $\sigma$ and $q^*(\sigma)$
in $G$ are the same, and hence their labels are the same.   \end{proof}

\begin{figure}[htbp] 
   \centering
   \includegraphics[height=1.2in]{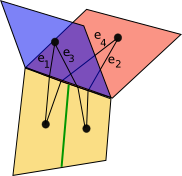} 
   \caption{A switching locus.}
   \label{match-fig}
\end{figure}

\begin{lemma}\label{extracycles-lem}  Let  $D'$ be obtained from $D$ by an allowable vertical subdivision on
a single $2$--cell.   The set of edges of each $\sigma \in C_{D'} \setminus q^*(C_D)$ contains exactly
one matched pair.
\end{lemma}

\begin{proof} Since $\sigma' \notin q^*(C_D)$, the quotient map $q$ is not injective on $\sigma'$.  Thus 
$q(\sigma')$ must contain two distinct edges $\eps_1,\eps_2$ with endpoint $v$, and these have lifts
$\eps_1'$ and $\eps_2'$ on $\sigma'$.  Since $\sigma'$ is a cycle, $\eps_1'$ and $\eps_2'$ must have
distinct endpoints, hence one is $v_1'$ and one is $v_2'$.   There cannot be more than one matched
pair on $\sigma'$, since $\sigma'$ can pass through each $v_i'$ only once.
\end{proof}

\begin{definition}
Let  $D'$ be obtained from $D$ by an allowable vertical subdivision on
a single $2$--cell.  Let $v$ be the vertex corresponding to the subdivided cell, and let $v_1'$ and $v_2'$ be its lifts to $D'$. 

For any pair of edges $\eps_1',\eps_2'$ with endpoints at $v_1'$ and $v_2'$ and distinct initial points $w_1'$ and
$w_2'$, there is a corresponding pair of edges $\eta_1',\eta_2'$ from $w_1'$ to $v_2'$ and from $w_2'$ to $v_1'$.
Write
$$
 \mbox{op} \{\eps_1',\eps_2'\} =\{\eta_1',\eta_2'\}.
$$
We call the pair $\{\eps_1',\eps_2'\}$  a {\it matched pair}, and 
$\{\eta_1',\eta_2'\}$ its {\it opposite}.  (See Figure~\ref{match-fig}).  
\end{definition}

\begin{lemma} If $\sigma' \in C_{D'}$ contains a matched pair, the edge-path
obtained from $\sigma'$ by exchanging the matched pair with its opposite is a cycle.
\end{lemma}

\begin{proof}  It is enough to observe that the set of endpoints and initial points of a matched pair
and its opposite are the same.
\end{proof}

Define a map $\mathfrak r : C_{D'} \rightarrow C_{D'}$ be the map that sends each $\sigma \in C_{D'}$ to the
cycle obtained by
exchanging each appearance of a matched pair on $\sigma' \in C_{D'}$ with its opposite.

\begin{lemma}\label{switching-lem}  The map $\mathfrak r$  is a simplicial map
of order two
that preserves length and labels.  It also fixes the elements of $q^*(C_D)$, and changes
the parity of the size of elements in $C_{D'} \setminus q^*(C_D)$.
\end{lemma}

\begin{proof} The map $\mathfrak r$ sends cycles to cycles, and hence simplicies to simplicies.
Since $\mbox{op}$ has order 2, it follows that $\mathfrak r$ has order 2.  The total number of
vertices does not change under the operation $\mbox{op}$.  It remains to check that the homology
class of $\sigma'$ and $\mathfrak r(\sigma')$ as embedded cycles in $X$  are the same, and that
the size switches parity.

There are two cases.  Either the matched edges lie on a single simple  cycle $\gamma'$ or on different
simple cycles  $\gamma_1',\gamma_2'$ on $\sigma'$.

In the first case, $\mathfrak r(\{\gamma'\})$ 
is a cycle with 2 components  $\{\gamma_1',\gamma_2'\}$.  As one-chains we have
\begin{eqnarray}\label{match-eqn}
\beta &=& \mathfrak r(\sigma') - \sigma' = \gamma_1' + \gamma_2' - \gamma' = \eta_1' + \eta_2' - \eps_1'- \eps_2'.
\end{eqnarray}
In $X$, $\beta$ bounds a disc (see Figure~\ref{match-fig}), thus $g(\gamma') = g(\gamma_1')+ g(\gamma_2')$,
and hence
\begin{eqnarray}\label{labels-eqn}
g(\sigma') &=& g(\mathfrak r(\sigma')).
\end{eqnarray}
The one component cycle $\gamma'$ is replaced by two simple cycles $\gamma_1'$ and $\gamma_2'$, and
hence the size of $\sigma'$ and $\mathfrak r(\sigma')$ differ by one.

Now suppose $\sigma'$ contains two cycles $\gamma_1'$ and $\gamma_2'$,
one passing through $v_1'$ and the other passing through $v_2'$.  
Then $\mathfrak r(\sigma')$ contains a simple cycle $\gamma'$  in place of $\gamma_1' + \gamma_2'$,
so the size decreases by one. By (\ref{match-eqn}) we have (\ref{labels-eqn}) for $\sigma'$ of this type. 
\end{proof}

\begin{proof}[Proof of Proposition \ref{verticalsubdivision-prop}] By Lemma~\ref{unique-lift-lem},
the quotient map $q: D' \rightarrow D$ induces
an injection of $q^* : C_D \hookrightarrow C_{D'}$ defined by the lifting map, and this map preserves labels.
We thus have
$$
\theta_{X,\mathfrak C,\psi} = 1 + \sum_{\sigma \in C_D} (-1)^{|\sigma|} g(\sigma)^{-1} = 1 + \sum_{\sigma' \in q^*(C_{D})} (-1)^{|\sigma'|} g(\sigma')^{-1}.
$$
The cycles in $C_{D'} \setminus q^*(C_D)$ partition into $\sigma', \mathfrak r(\sigma')$,  and
by Lemma~\ref{switching-lem}  the contributions of these pairs in $\theta_{X,\mathfrak C',\psi}$ cancel with each other. 
Thus, we have
$$
\theta_{X,\mathfrak C', \psi} =1 + \sum_{\sigma' \in C_{D'}} (-1)^{|\sigma'|} g(\sigma')^{-1} =  1 + \sum_{\sigma' \in q^*(C_{D})} (-1)^{|\sigma'|} g(\sigma')^{-1} = \theta_{X,\mathfrak C,\psi}.
$$
\end{proof}

\begin{definition}
Let $(X,\mathfrak C,\psi)$ be a branched surface and $c$ a $2$--cell. Let $p,q$ be two points on the boundary $1$-chain
$\partial c$ of $c$
that do not lie on the same $1$--cell of $\mathfrak C$.   
Assume that $p$ and $q$ each have the property that
\begin{enumerate}[(i)]
\item  it lies on a vertical
edge, or 
\item its forward flow under $\psi$ eventually lies on  a vertical $1$--cell of $(X,\mathfrak C)$.
\end{enumerate}
The {\it transversal subdivision} of $(X,\mathfrak C,\psi)$
at   $(c;p,q)$
is the new branched surface $(X,\mathfrak C',\psi)$ obtained from $\mathfrak C$ 
by doing the (allowable) vertical subdivisions of $\mathfrak C$ defined by $p$ and $q$,
and doing the additional subdivision induced by adding a $1$--cell from $p$ to $q$. 
\end{definition}

\begin{lemma}\label{trans-sub} Let $(X,\mathfrak C,\psi)$ be a branched surface, and
let $(X,\mathfrak C',\psi)$ be a transversal subdivision.  Then the corresponding cycle functions are the same.
\end{lemma}

\begin{proof}  
By first vertically subdividing $\mathfrak C$ along the forward orbits of $p$ and $q$ if necessary, 
 we may assume that $p$ and $q$ lie on different vertical $1$--cells on the boundary of $c$. 
 Let $v$ be the vertex of $D$ corresponding to $c$.  Then $D'$ is obtained from $D$ by substituting the vertex $v$ by
a pair $v_1', v_2'$ that are connected by a single edge. Each edge $\eps$ from $w\neq v$ to $v$ is replaced by 
an edge $\eps'$ from $w'$ to $v_1'$ and edge $\eps$ from $v$ to $u\neq v$ is replaced by an edge from $v_2'$ 
to $u'$. Each edge from $v$ to itself is substituted by an edge from $v_2$ to $v_1$. 
The cycle complexes  of $D$ and $D'$ are the same, and their labelings are identical. Thus the cycle function is preserved. 
\end{proof}

\subsection{Folding}\label{folding-sec} 

Let $(X,\mathfrak C,\psi)$ be a branched surface with a flow. 
Let  $c_1$ and $c_2$ be two 
cells with the property that their boundaries $\partial c_1$ and $\partial c_2$ both contain the segment 
$e_1e_2$, where $e_1$ is a vertical $1$--cell and $e_2$ is a transversal $1$--cell of $\mathfrak C$. 
Let $p$ be the initial point of $e_1$ and $q$ the end point of $e_2$.  Then $p$ and $q$ both lie on
vertical $1$--cells, and hence $(c_1;p,q)$ and $(c_2;p,q)$ define a composition of
 transversal subdivisions $\mathfrak C_1$ of $\mathfrak C$. 
For $i=1,2$, let  $e_3^i$,  be the new $1$--cell on $c_i$, 
and let  $\Delta(e_1,e_2,e_3^i)$ be the triangle $c_i$ bounded by the $1$--cells
 $e_1,e_2$ and $e_3^i$. 
 
 \begin{definition} 
The quotient map $F: X \rightarrow X'$  that identifies  $\Delta(e_1,e_2,e_3^1)$ and $\Delta(e_1,e_2,e_3^2)$
 (see Figure~\ref{folding-fig}) is called the {\it folding map} of $X$.   The quotient $X'$ is endowed with the
 structure of a branched surface $(X',\mathfrak C',\psi')$ induced by $(X,\mathfrak C_1,\psi)$.
 \end{definition}

 \begin{figure}[htbp] 
   \centering
   \includegraphics[height=1in]{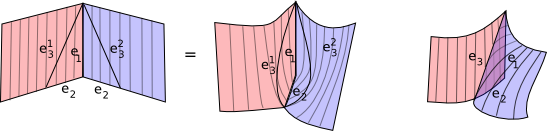} 
   \caption{The left and middle diagrams depict the two $2$--cells sharing the edges $e_1$ and $e_2$; the right diagram is the result of folding.}
   \label{folding-fig}
\end{figure}

 The following Proposition is easily verified (see Figure~\ref{folding-fig}).

\begin{proposition}
The quotient map $F$ associated to a folding
 is a homotopy equivalence, and the semi-flow $\psi: X \times \RR_+ \to X$ induces a semi-flow $\psi': X \times \RR_+ \to X$. 
\end{proposition}

\begin{definition}  Given a folding map $F : X \rightarrow X'$, there is an induced branched surface structure $(X',\mathfrak C', \psi')$
on $X$ given by taking the minimal cellular structure on $X'$ for which the map $F$ is a cellular map and deleting the image of $e_2$
if there are only two hinges containing $e_2$ on $X$. 
\end{definition}

\begin{remark}
In the case that $c_1,c_2$ are the only cells above $e_2$, then folding preserves the dual digraph $D$.
\end{remark}

 \begin{figure}[htbp] 
   \centering
   \includegraphics[height=1.2in]{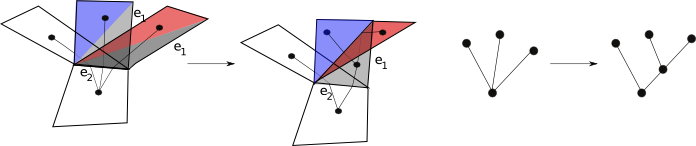} 
   \caption{Effect of folding on the digraph.}
    \label{foldingdigraph-fig}
\end{figure}

\begin{lemma}\label{folding-preserve-lem}
Let $F : X \rightarrow X'$ be a folding map, and let $(X',\mathfrak C', \psi')$ be the
induced branch surface structure of the quotient.  Then 
$$
\theta_{X,\mathfrak C,\psi} = \theta_{X',\mathfrak C',\psi'}.
$$
\end{lemma}

\begin{proof}
Let $D$ be the dual digraph of $(X,\mathfrak C,\psi)$ and $D'$ the dual digraph of $(X',\mathfrak C',\psi')$.
Assume that there are at least three hinges containing $e_2$.
Then $D'$ is obtained from $D$ by gluing 
two adjacent half edges (see Figure~\ref{foldingdigraph-fig}), a homotopy equivalence. 
Thus, $\C_\D^G = \C_{\D'}^G$, and  the cycle polynomials are equal.
\end{proof}

\section{Branched surfaces associated to a free group automorphism}\label{mappingtori_and_branchedsurfaces-sec}

Throughout this section, let $\phi \in \Out(F_n)$ be an element that can be represented
by an expanding irreducible train-track map $f : \tau \rightarrow \tau$.    
Let $\Gamma = F_n \rtimes_\phi \ZZ$, and $G= \Gamma^{\text{ab}}/\text{torsion}$. 
We shall define the mapping torus $(Y_f,\mathfrak C,\psi)$ associated $f$, and prove that its cycle polynomial $\theta_{Y_f,\mathfrak C,\psi}$ has a distinguished factor $\Theta$ with a distinguished McMullen cone $\calT$.
We show that the logarithm of the
house of $\Theta$ specialized at integral elements in the cone extends to a homogeneous of degree -1, real analytic concave function $L$
on an open cone in $\Hom(G,\RR)$, and satisfies a universality property.  

\subsection{Free group automorphisms and train-tracks maps}\label{freegroupautomorphisms-sec}
In this section we give some background definitions for free group automorphisms, 
and their associated train-tracks following \cite{DKL}.  We also recall some sufficient conditions
for a free group automorphism to have an expanding irreducible train-track map due
to work of Bestvina-Handel \cite{BH}.

\begin{definition}\label{graph-def}  A {\it topological graph} is a finite 1-dimensional cellular complex.
For each edge $e$, an orientation on $e$ determines an {\it initial} and {\it terminal} point of $e$.
Given an oriented edge $e$, we denote by $\overline e$, the edge $e$ with opposite orientation.
Thus the initial and terminal points of $e$ are respectively the terminal and initial points of $\overline e$.
An {\it edge path} on a graph is an ordered sequence of edges $e_1\cdots e_\ell$, where the endpoint of
$e_i$ is the initial point of $e_{i+1}$, for $i=1,\dots,\ell-1$.  The edge path has {\it back-tracking} if
$e_i = \overline{e_{i+1}}$ for some $i$.  The length of an edge path $e_1\cdots e_\ell$
is $\ell$.
\end{definition}

\begin{definition}\label{graphmap-def}
A {\it graph map} $f: \tau \to \tau$ is a continuous map from a  graph $\tau$ to itself that
sends vertices to vertices, and is a local embedding on edges.  A graph map assigns to each
edge $e \in \tau$ an edge path $f(e) = e_1\cdots e_\ell$ with no back tracks.  Identify 
the fundamental group $\pi_1(\tau)$ with a free group $F_n$.
A graph map $f$ {\it represents} an element
 $\phi \in \Out(F_n)$ if $\phi$ is conjugate to $f_*$ as an element of $\Out(F_n)$. 
\end{definition}

\begin{remark} In many definitions of graph map one is also allowed to collapse an edge, but for
this exposition, graph maps send edges to non-constant edge-paths.
\end{remark}

\begin{definition}\label{traintrack-def} 
 A graph map $f \from \tau \rightarrow \tau$ is a {\it train-track map} if
 \begin{enumerate}[(i)]
 \item $f$ is a homotopy equivalence, and
 \item  $f^k$ has no {\it back-tracking} for all $k \ge 1$, that is, for any $k \ge 1$,
and edge $e$, $f^k(e)$
is an edge path with no back-tracking.
\end{enumerate}
\end{definition}

\begin{definition} Given a train-track map $f \from \tau \to \tau$, let  $\{e_1,\dots,e_k\}$ be an
ordering of the edges of $\tau$, and let $D_f$ be the digraph
whose vertices $v_e$ correspond to the undirected edges $e$ of $\tau$, and whose edges  from $e_i$ to $e_j$ 
correspond to each appearance of $e_j$ and $\overline{e_j}$ in the edgepath $f(e_i)$.   The {\it transition matrix} $M_f$
 of $D_f$ is the directed adjacency matrix
 $$
 M_f  = [a_{i,j}],
 $$
 where $a_{i,j}$ is equal to the number of edges from $v_{e_i}$ to $v_{e_j}$.
 \end{definition}
 
 \begin{definition} \label{dil-def} If $f : \tau \rightarrow \tau$ be a  train-track map, the {\it dilatation}
of $f$ is given by the spectral radius of $M_f$
$$
\lambda(f) = \max\{|\mu| \ | \ \mbox{$\mu$ is an eigenvalue of $M_f$}\}.
$$
\end{definition}
 
 \begin{definition}
A train-track map $f : \tau \rightarrow \tau$ is {\it irreducible} if its transition matrix $M_f$ is irreducible, it
 is {\it expanding} if the lengths of edges of $\tau$ under iterations of $f$ are unbounded.
\end{definition}

\begin{remark} A Perron-Frobenius matrix is irreducible and expanding, but the converse is not necessarily true. 
\end{remark}

\begin{example}\label{irrNotPFexample} Let $\tau$ be the rose with four petals $a,b,c$ and $d$. 
Let $f : \tau  \rightarrow \tau$ be the train-track map associated to
the free group automorphism
\begin{equation}\label{myphi}
\begin{array}{rcl}
a &\mapsto& cdc\\
b &\mapsto& cd\\
c &\mapsto& aba\\
d &\mapsto& ab
\end{array}
\end{equation}
The train-track map $f$ has transition 
matrix
$$
M_f = \left [
\begin{array}{cccc}
0 & 0 & 2 & 1\\
0 & 0 & 1 & 1\\
2 & 1 & 0 & 0\\
1 & 1 & 0 & 0\\
\end{array}
\right ],
$$
which is an irreducible matrix, and hence $f$ is irreducible.  The train-track map is expanding, since its 
square is block diagonal, where each
block is a $2 \times 2$ Perron-Frobenius matrix.   On the other hand, $f$ is clearly not PF, since no power
of $M_f$ is positive.
\end{example}

\begin{definition}\label{wordlength-def}  Fix a generating set $\Omega = \{ \omega_1,\dots,\omega_n\}$ of $F_n$.
Then each $\gamma \in F_n$ can be written as a {\it word} in $\Omega$, 
\begin{eqnarray}\label{gamma-word-eqn}
\gamma = \omega_{i_1}^{r_1} \cdots \omega_{i_\ell}^{r_\ell}
\end{eqnarray}
where $\omega_{i_1}, \dots, \omega_{i_\ell} \in \Omega$ and $r_j \in \{1,-1\}$.
This representation is {\it reduced} if there are no cancelations, that is
 $\omega_{i_j}^{r_j} \neq \omega_{i_{j+1}}^{-r_{j+1}}$ for $j=1,\dots,\ell-1$.
 The word length $\ell_\Omega(\gamma)$ is the length $\ell$ of a reduced 
 word representing $\gamma$ in $F_n$.
The {\it cyclically reduced word length} $\ell_{\Omega,\mbox{\tiny cyc}}(\gamma)$
of $\gamma$ represented by the word in  (\ref{gamma-word-eqn})
is  the minimum word length of the elements
$$
\gamma_j = \omega_{i_j}^{r_j} \omega_{i_{j+1}}^{r_{j+1}} \cdots \omega_{i_\ell}^{r_\ell} \omega_{i_1}^{r_1} \cdots \omega_{i_{j-1}}^{r_{j-1}},
$$
for $j= 1,\dots,\ell-1$.
\end{definition}

\begin{proposition}\label{freedil-prop} Let $\phi \in \Out(F_n)$ be represented by an
expanding irreducible train-track map $f$,  and  let $\gamma \in F_n$ be
a nontrivial element.  Then either $\phi$ acts periodically on the  conjugacy class of $\gamma$ in $F_n$,
or the growth rate satisfies
$$
\lambda_{\Omega, \mbox{\tiny cyc}}(\gamma) =  \lim_{k} \ell_{\Omega,\mbox{\tiny cyc}}(\phi^k(\gamma))^\frac{1}{k} = \lambda(f),
$$
 and in particular, it is independent of the
choice of generators, and of $\gamma$.
\end{proposition}

\begin{proof} See, for example, Remark 1.8 in \cite{BH}.
\end{proof}

In light of Proposition~\ref{freedil-prop}, we make the following definition.

\begin{definition}\label{freedil-def} Let $\phi \in \Out(F_n)$ be an element that is
represented by an expanding irreducible train-track map $f$.
Then we define the {\it dilatation} of $\phi$ to be
$$
\lambda(\phi) = \lambda(f).
$$
\end{definition}

\begin{remark}\label{hyperbolic-def} An element $\phi \in \Out(F_n)$ is {\it hyperbolic} if
$F_n \rtimes_\phi \Z$ is word-hyperbolic.  It is  {\it atoroidal} 
 if there are no periodic conjugacy classes of elements of $F_n$ under iterations of $\phi$.
By a result of Brinkmann \cite{Brinkmann} (see also \cite{BF92}),  $\phi$ is hyperbolic if and only if $\phi$ is atoroidal.
\end{remark}

\begin{definition}\label{fully-irreducible-def} An automorphism $\phi \in \Out(F_n)$ is {\it reducible} if  $\phi$ leaves the
conjugacy class  of a proper free factor in $F_n$ fixed.
If $\phi$ is not reducible it is called {\it irreducible}.   If $\phi^k$ is irreducible for all $k \ge 1$, then
 $\phi$ is {\it fully irreducible}.
 \end{definition}

\begin{theorem}[Bestvina-Handel \cite{BH}] \label{freeautodil-thm}  If $\phi \in \Out(F_n)$ is irreducible, then
$\phi$ can be represented by an irreducible train track map, and
if $\phi$ is
fully irreducible, then it can be represented by a PF train track map.
\end{theorem}

\begin{remark}
Theorem \ref{main-thm} deals with an automorphism $\phi$ that can be represented by an irreducible 
and expanding train-track map. It does not follow that for such an automorphism every train-track representative 
is expanding and irreducible. For example, consider the automorphism $\phi$ from Example 
\ref{irrNotPFexample}. Let $\tau'$ be a graph constructed from an edge $e$ with two distinct endpoints 
$v$ and $w$ by attaching at $v$ two loops labeled $a$ and $b$ and attaching at $w$ two loops $c$ and $d$. The map $f' \from \tau' \to \tau'$ defined by equation \ref{myphi} and $e \mapsto \bar e$ represents the same automorphism $\phi$ as in Example \ref{irrNotPFexample}. However, since $e$ is invariant, the map is not irreducible and not expanding. \\
If we assume that $\phi$ is fully irreducible, then all train-track representatives are expanding. Indeed, let $f' \from \tau' \to \tau'$ be a train-track representative of $\phi$. Then $f'$ is irreducible because an invariant subgraph will produce a $\phi$-invariant free factor. It is now enough to show that some edge is expanding. Let $\al$ be an embedded loop in $\tau'$. We can think of $\al$ as a conjugacy class in $F_n$. Then by proposition  \ref{freedil-prop} either $\al$ is periodic or $\al$ grows exponentially. However, $\al$ cannot be periodic since $\al$ represents a free factor  of $F_n$. Therefore, $\al$ grows exponentially, hence some edge grows exponentially and because $f'$ is irreducible, all edges grow exponentially. 
\end{remark}

\subsection{The mapping torus of a train-track map.}\label{mappingtorus-sec}
In this section we define the branched surface $(X_f,\mathfrak C_f,\psi_f)$ associated to
 an irreducible expanding train-track map $f$.
 
\begin{definition}\label{mappingtorus-def} The {\it mapping torus}  $(Y_f,\psi_{f})$
associated to $f: \tau \rightarrow \tau$ is the branched surface where
 $Y_f$ is the quotient of  $\tau \times [0,1]$ by the 
identification $(t,1) \sim (f(t),0)$, and $\psi_f$ is the semi-flow induced by
the product structure of $\tau \times [0,1]$.  
Write 
$$
q : \tau \times [0,1] \rightarrow Y_f
$$
 for the quotient map.
The map to the circle induced by projecting $\tau \times [0,1]$
to the second coordinate induces a map $\rho : Y_f \rightarrow S^1$.
\end{definition}

\begin{definition}\label{MTcelldecomp-def}
The $\psi_f$-compatible cellular decomposition $\mathfrak C_f$ for $Y_f$ 
is defined as follows.
For each edge $e$, let $v_e$ be the initial
vertex of $e$ (the edges $e$ are oriented by the orientation on $\tau$).  The $0$--cells 
of $\mathfrak C_f$ are $q(v_e \times \{0\})$,
 the $1$--cells are of the form
 $s_e = q(v_e \times [0,1])$ or $t_e = q(e \times \{0\})$, and the $2$--cells are $c_e = q(e \times [0,1])$, where
 $e$ ranges over the oriented edges of $\tau$. 
 For this cellular decomposition of $Y_f$, the collection $\calV$ of $s_e$ is  the set of {\it vertical} $1$--cells 
 and the collection $\calE$ of $1$--cells $t_e$ is the set of {\it horizontal} $1$--cells.
  \end{definition}
 
 By this definition $(Y_f,\mathfrak C_f, \psi_f)$ is a branched surface.  Let $\theta_{Y_f, \mathfrak C_f,\psi_f}$
 be the associated cycle function (Definition~\ref{cyclefunction-def}).
 
 \begin{figure}[htbp] 
   \centering
   \includegraphics[height=1.2in]{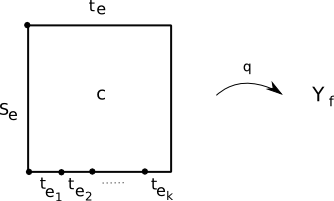} 
   \caption{A cell of the mapping torus of a train-track map.}
   \label{fig:mappingtorus}
\end{figure}

\begin{proposition}
The digraph $D_f$ for the train-track map $f$ and the dual digraph
of $(Y_f,\mathfrak C_f, \psi_f)$ are the same, and we have
$$
\lambda(\phi) = |\theta_{Y_f, \mathfrak C_f,\psi_f}^{(\alpha)}|,
$$
where $\alpha : \Gamma \rightarrow \Z$ is the projection associated to $\phi$.
\end{proposition}

\begin{proof}
Each $2$--cell $c$ of $(Y_f,\mathfrak C_f, \psi_f)$ is the quotient of one 
 drawn as in Figure~\ref{fig:mappingtorus}, and hence there is a one-to-one correspondence
 between $2$--cells and edges of $\tau$.   One can check that for each
time $f(e)$ passes over the edge $e_i$, there is a corresponding hinge between the cell $q(e \times [0,1])$ and the
cell $q(e_i \times [0,1])$.  This gives a one-to-one correspondence between the directed edges of $D_f$
and the edges of the dual digraph.

Recall that $\lambda(\phi) = \lambda(f)$ is the spectral radius of $M_f$
(Definition~\ref{freedil-def}).  By Theorem~\ref{coeff-thm},
the characteristic polynomial of $D_f$ satisfies
$$
P_{D_f} (x) = x^m \theta_{D_f}(x).
$$
Each edge of $D_f$ has length one with respect to the map $\alpha$, and hence for each cycle $\sigma \in C_{D_f}$,
the number of edges in $\sigma$ equals $\ell_{\alpha}(\sigma)$.
It follows that $\theta_{D_f}(x)$ is the specialization by $\alpha$ of the cycle function $\theta_{Y_f,\mathfrak C_f,\psi_f}$, and
we have
\[
\lambda_{\mbox{PF}}(D_f) = |P_{D_f}| = |\theta_{D_f}| = |\theta_{Y_f,\mathfrak C_f,\psi_f}^{(\alpha)}|. \qedhere
\]
\end{proof}

In the following sections, we study the behavior of $|\theta_{Y_f,\mathfrak C_f,\psi_f}^{(\alpha)}|$ as we let $\alpha$
vary.

\subsection{Application of McMullen's theorem to cycle polynomials}

Fix a train-track map $f : \tau \rightarrow \tau$.
Recall that  $\theta_f = \theta_{Y_f,\mathfrak C_f,\pi_f} = 1 + \sum_{\sigma \in \calC_{D_f}^G} (-1)^{|\sigma|} g(\sigma)^{-1}$.  
Thus 
the McMullen cone $\calT_{\theta_f}(1)$
is  given by
\begin{eqnarray*}
\calT_{\theta_f} (1) &=& \{ \alpha \in \Hom(G;\R)\ | \ \alpha (g) > 0, \ \mbox{for all $g \in \Supp(\theta)$}\}\\
&=& \{ \alpha \in \Hom(G;\R)\ | \ \alpha(g) > 0,\ \mbox{for all $g\in G$ such that $a_g \neq 0$}\}.
\end{eqnarray*}
(see Definition~\ref{McMullen_cone-def}).   We write $\calT_f = \calT_{\theta_f}(1)$ for simplicity
when the choice of cone associated to $\theta_f$ is understood.

\begin{proposition}\label{delta-prop} Let $\calT_f$ be the McMullen cone for $\theta_f$.
The map 
$$
\delta : \Hom(G;\R)  \rightarrow \R 
$$
defined by
$$
\delta(\alpha) =  \log |\Theta^{(\alpha)}|,
$$
extends to a homogeneous of degree $-1$, real analytic, convex function on $\calT_f$ that goes to
infinity toward the boundary of affine planar sections of $\calT_f$.
Furthermore, $\theta_f$ has a factor $\Theta$ with the properties:
\begin{enumerate}
\item  for all $\alpha \in \calT_f$,
$$
 |\theta_f^{(\alpha)}| = |\Theta^{(\alpha)}|,
$$
 and
\item {\it minimality}: if  $\theta \in \Z G$ satisfies $|\theta^{(\alpha)}| = |\theta_f^{(\alpha)}|$ for
all $\alpha$ ranging among the integer points of an open subcone of $\calT_f$,
then $\Theta$ divides $\theta$.
\end{enumerate}
\end{proposition}

To prove Proposition~\ref{delta-prop} we write $G = H \times \langle s \rangle$ and identify $\theta_f$ with the
characteristic polynomial $P_f$ of an expanding $H$-matrix $M_f$.  Then 
Proposition~\ref{delta-prop} follows from Theorem~\ref{McM-thm}. 

Let 
$$
G=H_1(Y_f;\Z)/\mbox{torsion} = \Gamma^{\mbox{ab}}/{\mbox{torsion}},
$$
  and let $H$ be the image 
of $\pi_1(\tau)$ in $G$ induced by the composition
$$
\tau \rightarrow \tau\times \{0\} \hookrightarrow \tau \times [0,1] \overset{q}\rightarrow Y_f.
$$
Let 
$\rho_* : G \rightarrow \Z$ be the map corresponding to  $\rho : Y_f \rightarrow S^1$.

\begin{lemma} The group $G$ has decomposition as $G = H \times \langle s \rangle$, where
$\rho_*(s) = 1$.
\end{lemma}

\begin{proof} The map $\rho_*$ is onto $\Z$ and its kernel equals $H$.    Take any $s \in \rho_*^{-1}(1)$.  Then
since $s \notin H$, and $G/H$ is torsion free, we have $G = H \times \langle s \rangle$.
\end{proof}

We call $s$ a {\it vertical generator} of $G$ with respect to $\rho$, and identify $\Z G$ with the ring of
Laurent polynomials $\Z H (u)$ in the variable $u$ with coefficients in $\Z H$, by the map $\Z G \rightarrow \Z H(u)$
determined by sending $s \in \Z G$ to $u \in \Z H (u)$.  

\begin{definition} Given $\theta \in \Z G$, the {\it associated polynomial} $P_\theta(u)$ of $\theta$ is the image of
$\theta$ in $\Z H(u)$ defined by the identification $\Z G = \Z H(u)$.
\end{definition}

The definition of support for an associated polynomial $P_\theta$ is analogous to the one for $\theta$.

\begin{definition} The {\it support} of an element $P_\theta \in \Z H (u)$ is given by
$$
\Supp(P_\theta) = \{ h u^r\  |\ \mbox{$h$,$r$ are such that $(h,s^r) \in \Supp(\theta)$} \}.
$$
\end{definition}

Let $P_{\theta_f} \in Z H (u)$ be the polynomial associated to $\theta_f$.
Instead of realizing $P_{\theta_f}$ directly as a characteristic polynomial of an $H$-labeled digraph, we  start with a more natural
labeling of the digraph $D_f$.

Let $C_1=\Z^{\calV \cup \calE}$ be the free abelian
group generated by the  positively oriented edges of $Y_f$, which we can also think of as 1-chains in $\mathfrak C^{(1)}$
(see Definition~\ref{MTcelldecomp-def}).  
Let  $Z_1 \subseteq C_1$ be the subgroup corresponding to closed 1-chains.  
The map $\rho$ induces a homomorphism $\rho_* : C_1 \rightarrow \Z$.

Let $\nu: Z_1 \rightarrow G$ be the quotient map.
The map $\nu$ determines a ring homomorphism
\begin{eqnarray*}
\nu_*: \Z Z_1&\rightarrow& \Z G\\
\sum_{g \in Z_1} a_g g &\mapsto& \sum_{g \in J}a_g \nu(g).
\end{eqnarray*}
This extends to a map from $\Z Z_1[u]$ to $\Z G[u]$. 

Let $K_1 \subseteq Z_1$ be the kernel of $\rho_*|_{Z_1} : Z_1 \rightarrow \Z$.  
Then $H$ is the subgroup of $G$ generated by $\nu (K_1)$. 
 Let $\nu^H$ be the restriction of $\nu$ to $K_1$.
Then $\nu^H$ similarly defines 
$$
\nu^H_*: \Z K_1 \rightarrow \Z H,
$$
 the restriction of $\nu_*$ to $\Z K_1$, and this
extends to 
$$
\nu^H_* : \Z K_1[u] \rightarrow \Z H[u].
$$

\begin{proposition}\label{key-prop}
There is a Perron-Frobenius $K_1$-matrix $M^{K_1}_f$, whose
 characteristic polynomial $ P^{K_1}_f(u) \in \Z K_1[u]$
satisfies
$$
P_{\theta_f}(u) =  u^{-m} \nu^H_*   P^{K_1}_f(u).
$$
\end{proposition}

To construct $M^{K_1}_f$, we define a $K_1$-labeled digraph
with underlying digraph $D_f$.
Let $s$ be a vertical generator relative to $\rho_*$.
Choose any element $s' \in Z_1$ mapping to the vertical generator $s \in G$.
Write each $s_e \in \calV$ as $s_e =  s' k_e$, where $k_e \in K_1$.
Label edges of the digraph $D_f$
 by elements of $C_1$ as follows.  Let $f(e) = e_1 \cdots e_r$.  Then for each 
 $i=1,\dots,r$, there is a corresponding hinge  $\kappa_i$ whose initial cell corresponds to $e$ and 
 whose terminal cell corresponds to $e_i$.  Take any edge $\eta$ on $D_f$ emanating from $v_e$.  Then
 $\eta$ corresponds to one of the hinges $\kappa_i$, and has initial vertex $v_e$ and terminal vertex
 $v_{e_i}$.   For such an $\eta$, define
\begin{eqnarray*}
g(\eta) &=&  s_e  t_{e_1}  \cdots t_{e_{i-1}}\\
 &=&  s' k_e t_{e_1} \cdots t_{e_{i-1}}\\
&=&  s' k(\eta)
\end{eqnarray*}
where $k(\eta) \in K_1$.  This defines a map from the edges $D_f$ to $C_1$ giving
a labeling $\D_f^{C_1}$.  It also defines a map from edges of $D_f$ to $K_1$
by $\eta \mapsto k(\eta)$.  Denote this labeling of $D_f$ by $D_f^{K_1}$.

\begin{definition}\label{conjugatedigraph-def} Given a labeled digraph $\Dlab^G$, with edge labels $g(\eps)$ for each edge $\eps$ of 
the underlying digraph $D$,
the {\it conjugate digraph} $\widehat \Dlab^G$ of $\Dlab^G$ is
the digraph with same underlying graph $D$, and edge labels $g(\eps)^{-1}$ for each edge $\eps$ of $D$.
\end{definition}

Let $\widehat \calD_f^{K_1}$ be the conjugate digraph of $\calD_f^{K_1}$, and
let $\widehat M^{K_1}_f$ be the directed adjacency matrix for $\widehat \calD_f^{K_1}$.

\begin{lemma}\label{conjugatedigraph-prop}  
The cycle function $\theta_f \in \Z G$ and
the characteristic polynomial $\widehat P_f (u) \in \Z K_1 [u]$ of $\widehat M^{K_1}$
satisfy
$$
 \nu_*^H (\widehat P_f(u)) = u^mP_{\theta_f}(u).
$$
\end{lemma}

\begin{proof}  By  the coefficient theorem for labeled digraphs (Theorem~\ref{CTlabeled-thm}) we have
\[
\widehat P_f (u)  = u^m\theta_{\calD_f^{K_1}} = u^m(1+ 
\sum_{\sigma \in C_{D_f}} (-1)^{|\sigma|} k(\sigma)^{-1} u^{-\ell(\sigma)}).
\]
Since $g(\sigma) = k(\sigma)s^{\ell(\sigma)}$, a comparison of $\widehat P_f$ with 
$\theta_f$ gives the desired result.
\end{proof}

\begin{proof}[Proof of Proposition~\ref{delta-prop}] Let $M_f$ be the matrix with entries in $\Z H$ 
given by taking $\widehat M^{K_1}$ and applying
$\nu^H$ to its entries.  Then the characteristic polynomial $P_f$ of $M_f$ is  related to 
the characteristic polynomial $\widehat P_f$ of $\widehat M^{K_1}$ by
$$
P_f(u) = \nu_*^H(\widehat P_f(u)).
$$
Thus,  Lemma~\ref{conjugatedigraph-prop} implies
$$
P_f (u) = u^mP_{\theta_f}(u),
$$
and hence the properties of Theorem~\ref{McMullen-thm} applied to $\widehat P_f$ also hold
for $\theta_f$.
\end{proof}

\section{The folded mapping torus and its DKL-cone}\label{foldedmappingtorus-sec}

We start this section by defining a folded mapping torus and stating some 
results of Dowdall-Kapovich-Leininger on deformations of free group automorphisms.
We then proceed to finish the proof of the main theorem. 

\subsection{Folding maps}  In \cite{Stallings_folds} Stallings introduced
the notion of a folding decomposition of a train-track map.

\begin{definition}[Stallings \cite{Stallings_folds}]\label{fold-def}  
Let $\tau$ be a topological graph, and $v$ a vertex on $\tau$.  Let $e_1,e_2$ be 
two distinct edges of $\tau$ meeting at $v$, and let $q_1$ and $q_2$ be
their other endpoints.  Assume that $q_1$ and $q_2$ are distinct vertices of $\tau$.
The  {\it fold} of $\tau$ at $v$, is the image $\tau_1$ of a quotient map $\f_{(e_1,e_2:v)} : \tau \rightarrow \tau_1$ where
$q_1$ and $q_2$ are identified as a single vertex in $\tau_1$
and the  two edges $e_1$ and $e_2$ are identified as a single edge in $\tau_1$.  
The map $\f_{(e_1,e_2:v)}$ is called a {\it folding map}
\end{definition}

It is not hard to check the following.

\begin{lemma}  
Folding maps on train-tracks are  homotopy equivalences.
\end{lemma}

\begin{definition}
A {\it folding decomposition} of a graph map $f \from \tau \rightarrow \tau$ is a decomposition
\[
f = h f_k\cdots f_1
\]
where $\tau_0$ is the graph $\tau$ with a finite number of subdivisions on the edges, 
$f_i \from \tau_{i-1} \rightarrow \tau_i$ for $i=1,\dots,k$ are folding maps, and 
$h \from \tau_k \rightarrow \tau_k$ is a homeomorphism.
We denote the folding decomposition by $(f_1,\dots,f_k;h)$.
\end{definition}

\begin{figure}[htbp] 
   \centering
   \includegraphics[scale=.7]{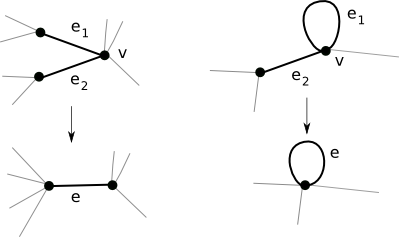} 
   \caption{Two examples of folding maps.}
   \label{foldfig}
\end{figure}

\begin{lemma}[Stallings \cite{Stallings_folds}] 
Every homotopy equivalence of a graph to itself has a (non-unique) folding decomposition. 
Moreover, the homeomorphism at the end of the decomposition is uniquely determined.
\end{lemma}

Decompositions of a train-track map into a composition of folding maps gives rise to a 
branched surface that is homotopy equivalent to $Y_f$.

Let $f: \tau \rightarrow \tau$ be a train-track map with a folding decomposition
$\f = (f_1,\dots,f_k;h)$, where $f_i : \tau_{i-1} \rightarrow \tau_i$ is a folding map, 
for $i=1,\dots,k$, $\tau = \tau_0 = \tau_k$,  and $h: \tau \rightarrow \tau$ is a 
homeomorphism.

For each $i=0,\dots,k$, define a 2-complex $X_i$ and semiflow $\psi_i$ as follows.
Say $f_i$ is the folding map on $\tau$ folding $e_1$ onto $e_2$ at their common 
endpoint $v$.  Let $q$ be the initial vertex of both $e_1$ and $e_2$, and $q_i$ the 
terminal vertex of $e_i$. Let $X_i$ be the quotient of $\tau_{i-1} \times [0,1]$ 
obtained by identifying the triangles 
\[
\big[(q,0),(q,1),(q_1,1) \big] \quad\text{on} \quad e_1 \times[0,1]
\]
with 
\[
\big[(q,0),(q,1),(q_2,1) \big] \quad\text{on} \quad e_2 \times [0,1].
\]
The semi-flow $\psi_i$ is defined by the second coordinate of $\tau_{i-1} \times [0,1]$. 
By the definitions, the image of $\tau_{i-1} \times\{1\}$ in $X_i$ under the quotient map
is $\tau_i$.

Let $X_\f$ be the union of pieces $X_0 \cup \cdots \cup X_k$ 
so that the image of $\tau_{i-1} \times \{1\}$ in $X_{i-1}$ is attached to the image of
$\tau_i \times \{0\}$ in $X_i$ by their identifications with $\tau_i$, and the image of 
$\tau_k \times \{1\}$
in $X_k$ is attached to the image of $\tau_0 \times \{0\}$ in $X_0$ by $h$.

Each $X_i$ has a semiflow induced by its structure as the quotient of 
$\tau_{i} \times [0,1]$. This
induces a semiflow $\psi_\f$ on $X_\f$.  The cellular structure on $X_\f$ is defined 
so that the $0$--cells correspond to the images in $X_i$ of $(q,0), (q,1), (q_1,1)$ and 
$(q_2,1)$.   The transversal $1$--cells of $\mathfrak C_\f$ correspond to the images in 
$X_i$ of edges $[(q,0), (q_i,1)]$, for $i=1,2$. The vertical $1$--cells of $\mathfrak C_\f$ 
are the forward flows of all the vertices of $X_\f$.   The vertical
and transversal $1$--cells form the boundaries of the $2$--cells of $\mathfrak C_\f$.

\begin{definition}[cf. \cite{DKL}]\label{foldedmappingtorus-def} 
A {\it folded mapping torus} associated to a folding decomposition $\f$
of a train-track is the branched surface $(X_\f,\mathfrak C_\f, \psi_\f)$ defined above.
\end{definition}

\begin{lemma}  
If $(X_\f,\mathfrak C_{\f}, \psi_{\f})$ is a folded mapping torus, then there is a  cellular 
decomposition of $X_\f$ so that the following holds:
\begin{enumerate}[(i)]
\item The $1$--skeleton $\mathfrak C_\f^{(1)}$  is a union of oriented $1$--cells 
meeting only at their endpoints.
\item  Each $1$--cell has a distinguished orientation so that the corresponding tangent 
directions are either tangent to the flow ({\it vertical case}) or positive but skew to the 
flow ({\it diagonal case}).
\item The endpoint of any vertical $1$--cell is the starting point of another vertical $1$--cell.
\end{enumerate}
\end{lemma}

\begin{proof}  The cellular decomposition of $X_\f$ has transversal $1$--cells 
corresponding to the folds, and vertical $1$--cells corresponding to the flow suspensions 
of the endpoints of the diagonal $1$--cells. 
\end{proof}

\subsection{Simple example} We give a simple example of a train-track map, a 
folding decomposition and their associated branched surfaces.

Consider the train-track in Figure~\ref{fig:figure8}, and the train-track map corresponding 
to the free group automorphism $\phi \in \Out(F_2)$ defined by
\begin{eqnarray*}
a &\mapsto& ba\\
b &\mapsto& bab
\end{eqnarray*}

\begin{figure}[htbp] 
   \centering
   \includegraphics[height=0.5in]{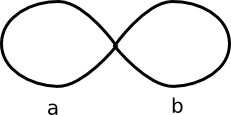} 
   \caption{Two petal rose. }
   \label{fig:figure8}
\end{figure}

Then the corresponding train-track map $f \from \tau \rightarrow \tau$ sends the edge $a$ 
over $b$ and $a$, and the edge $b$ over $b$ then $a$ then $b$.  The corresponding 
mapping torus is shown on the left of Figure~\ref{rosemap-fig}.    
\begin{figure}[htbp] 
   \centering
   \includegraphics[height=1.4in]{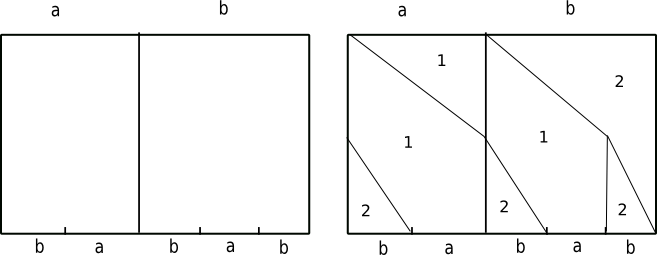} 
   \caption{Mapping torus and folded mapping torus. }
   \label{rosemap-fig}
\end{figure}
A folding decomposition is obtained from $f$ by subdividing the edge $a$ twice
and the edge $b$ three times.  The first fold identifies the entire edge $a$ with two 
segments of the edge $b$.   This yields a train-track that is homeomorphic to the original. 
The second fold identifies the edge $b$ to one segment of the edge $a$.   The resulting 
folded mapping torus is shown on the right of Figure~\ref{rosemap-fig}.
Here cells labeled with the same number are identified.

\subsection{Dowdall-Kapovich-Leininger's theorem}

Recall that elements $\alpha \in H^1(X_\f;\R)$ can be represented
by cocycle classes $z \from H_1(X_f;\R) \rightarrow \R$.

\begin{definition} \label{Def:DKL-cone}
Given a branched surface $\sX = (X_\f,\mathfrak C_\f,\psi_\f)$, orient the edges of $\mathfrak C_\f$ positively with respect to the semi-flow $\psi_\f$.
The associated {\it positive cone} for $\sX$ in $H^1(X;\R)$, denoted
$\calA_\f$, is given by
\[
\calA_\f = \big\{\alpha \in H^1(X_\f;\R) \ \big| \ 
 \mbox{there is a  $z \in \alpha$ so that $z(e) > 0$ for all  $e \in\mathfrak C_\f^{(1)}$} \big\}.
\]
\end{definition}

\begin{theorem}[Dowdall-Kapovich-Leininger \cite{DKL}]\label{DKL1-thm}
Let $f$ be an expanding irreducible train-track map, $\f$ a folding decomposition of $f$
and $(\xdkl,\cdkl,\pdkl)$ the folded mapping torus associated to $\f$.  For 
every integral $\al \in A_\f$ there is a continuous map $\eta_\al \from X_\f \to S^1$
with the following properties.
\begin{enumerate}
\item Identifying $\pi_1(\xdkl)$ with $\Gamma$ and $\pi_1(S^1)$ with $\ZZ$, 
$(\eta_\al)_* = \al$,
\item The restriction of $\eta_\al$ to a semiflow line is a local diffeomorphism. The restriction of $\eta_\al$ to a flow line in a $2$--cell is a non-constant affine map. 
\item\label{lenVSalpha} For all simple cycles $c$  in $\xdkl$ oriented positively with respect to the flow, $\ell(\eta_\al(c)) = \al([c])$ where $[c]$ is the image of $c$ in $G$.
\item\label{embedGraph} Suppose $x_0 \in S^1$ is not the image of any vertex, denote $\tau_\al := \eta_\al^{-1}(x_0)$. 
If $\al$ is primitive $\tau_\al$ is connected, and $\pi_1(\tau_\al) \cong \ker(\al)$.
\item For every $p \in \tau_\al \cap (\cdkl)^{(1)}$, there is an $s \geq 0$ so that $\psi(p,s) \in (\cdkl)^{(0)}$.
\item The flow induces a map of first return $f_\al\from \tau_\al \to \tau_\al$, which is an expanding  irreducible train-track map. 
\item The assignment that associtates to a primitive integral $\al \in A_\f$ the logarithm of the dilatation of $f_\al$ can be extended to a continuous and convex function on $A_\f$. 
\end{enumerate}
\end{theorem}

\begin{proof}
This is a compilation of results of \cite{DKL}. 
\end{proof}

\subsection{The proof of main theorem}
In this section, we prove Theorem~\ref{main-thm}. A crucial step to our proof is that 
the mapping torus $\sY= (Y_f,\mathfrak C_f,\psi_f)$
and the folded mapping torus $\sX=(X_\f,\mathfrak C_\f,\psi_\f)$ both have the same cycle 
polynomial.  

\begin{proposition} \label{Prop:same-plynoliam} 
The cycle functions $\theta_{\sY}$ of $(Y_f,\mathfrak C_f, \psi_f)$ and $\theta_{\sX}$ of
 $(X_\f,\mathfrak C_\f, \psi_\f)$ coincide. 
\end{proposition}

\begin{proof} 
We observe that  $(X_\f,\mathfrak C_\f, \psi_\f)$ can be obtained from the
mapping torus of the train-track map $(Y_f,\mathfrak C_f, \psi_f)$ by a sequence of folds, 
vertical subdivisions and transversal subdivision, as defined in Sections \ref{vertical-sec} 
and \ref{folding-sec}.  The reverse of these folds is shown in Figure~\ref{unfold-fig}.
\begin{figure}[htbp] 
   \centering
   \includegraphics[height=1.2in]{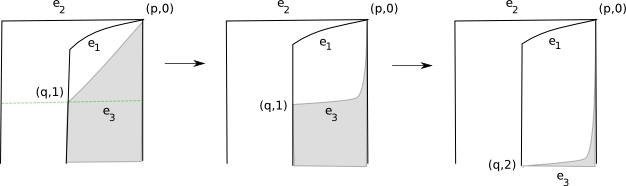} 
   \caption{Vertical unfolding. }
   \label{unfold-fig}
\end{figure}

\noindent
The proposition now follows from  Proposition~\ref{verticalsubdivision-prop}, 
Lemma~\ref{trans-sub} and Lemma~\ref{folding-preserve-lem},
\end{proof}

We also  need to check that our theorems apply for vectors in
the DLK-cone $\calA_\f$. 

\begin{proposition}\label{AinT}
Let $\theta_\f$ be the cycle polynomial of the DKL mapping torus.
Then
\[
\calA_\f\subseteq  \calT_{\theta_\f}(1).
\]
\end{proposition}
\begin{proof}
We need to show that, for every $\sig \in \C_{\xdkl}$ with 
$|\sig|=1$, we have $\al(g(\sig))>0$. Then for all nontrivial $g \in \Supp(\theta_\f)$, 
we have $\al(g)>0$, and hence $\alpha \in \calT_{\theta_\f}(1) = \calT$. 
Let $c$ be a closed loop in $D$. The embedding of $D$ in $\xdkl$ described in 
Def.~\ref{dualdigraph-def} induces an orientation on the edges of $D$ that is 
compatible with the flow $\psi$. For each edge $\mu$ of $c$, 
item (2) in Theorem \ref{DKL1-thm} implies 
$\ell(\eta_\al(\mu))>0$ and tem (3)  in Theorem \ref{DKL1-thm} implies
$\al([c]) = \ell(\eta_\al(c)) = \sum_{\mu \in c} \ell(\eta_\al(\mu))>0$.
\end{proof}

\begin{proposition}\label{dil-prop} Let $(\xdkl,\cdkl, \pdkl)$ be the folded mapping torus, $\theta_\f$ its cycle  polynomial and $\calA_\f$ the DKL-cone. For all primitive integral $\alpha \in A_\f$,  we have 
\[ \lambda(\phi_\alpha) = |\theta_\f^{(\alpha)}|\]
\end{proposition}

\begin{proof}
Embed $\tau_\al$ in $\xdkl$ transversally as in Theorem \ref{DKL1-thm}(\ref{embedGraph}), and perform a vertical subdivision so that the intersections of $\tau_\al$ with $(\xdkl)^{(1)}$ are contained in the $0$--skeleton (we can do this by Theorem \ref{DKL1-thm}(5)). 
Perform transversal subdivisions to add the edges of $\tau_\al$ to the 1-skeleton. Then perform a sequence 
of foldings and unfoldings to move the branching of the complex into $\tau_\al$, and remove the extra 
edges. Denote the new branched surface
 by $X_\f^{(\al)} = (\xdkl,\cdkl^{(\alpha)},\pdkl)$. These operations preserve the cycle polynomials of the respective 2-complexes, therefore we denote all of these polynomials by $\theta$ (in particular $\theta_\f = \theta$).

Let $f_\al \from \tau_\alpha \rightarrow \tau_\alpha$ 
 be the map induced by the first return map, 
and $D_\al$ its digraph. Then $f_\alpha$ defines a train-track map representing $\phi_\alpha$,
and $\lambda(\phi_\alpha) = \lambda(D_\alpha)$.

The (unlabeled) digraph $D_\f^{(\alpha)}$ of the new branched surface $(\xdkl,\cdkl^{(\alpha)},\pdkl)$ is identical to $D_\al$.
For every cycle $c$ in $D_\al$,  
\[ \begin{array}{ll}
\ell(c) & = \# \text{edges in }c = \# \tau_\al \cap (\xdkl^{(\al)})^{(1)} \\[0.1 cm]
 &= \ell(\eta_\al(c))=\al([c]).
\end{array}\]
The equalities in the second line follow from Theorem \ref{DKL1-thm}(\ref{embedGraph}) 
and (\ref{lenVSalpha}) respectively.  
Thus $\ell(\sig) = \al(g(\sig))$ for every $\sig \in C_{D_\al}$. Let $P_\al(x)$ be the characteristic polynomial of the incidence matrix associated to $D_\al$. By the coefficients theorem for digraphs  (Theorem~\ref{coeff-thm}) we have:
\begin{align*}
P_\al(x) &= x^m + \sum_{\sig \in C_D} (-1)^{|\sig|}x^{m-\ell(\sig) } \\
&= x^m\left(1 +\sum_{\sig \in C_{D_\al}} (-1)^{|\sig|}x^{\al(g(\sig))} \right) \\
&= x^m \theta^{(\al)}
\end{align*} 
Therefore,
\[ 
\lambda(\phi_\al) = |P_\al| = |\theta^{(\al)}|  \qedhere
\]
\end{proof}

We are now ready to prove our main result.

\begin{proof}[Proof of Theorem~\ref{main-thm}]
Choose an expanding train-track representative $f$ of $\phi$, and a folding decomposition $\f$ of $f$. 
As before, let $\sY = (Y_f, \mathfrak C_f,\psi_f)$  be the mapping torus of $f$, and
$\sX = (X_{\f}, \mathfrak C_{\f},\psi_\f)$  the folded mapping torus. 
By \propref{same-plynoliam} their cycle function $\theta_{\sY}, \theta_{\sX}$ are equal,
and we will call them $\theta$.

Let $\Theta$ be the minimal factor of $\theta$ defined in Proposition \ref{delta-prop}, 
and let $\calT= \calT_{\Theta}(1)$ be the McMullen cone.
By Proposition~\ref{AinT}, $A_\f \subseteq \calT$, and by Proposition 
\ref{dil-prop}, $\lambda(\phi_\al) = |\Theta^{(\al)}|$. By Proposition \ref{delta-prop}, 
$|\Theta_\phi^{(\al)}| = |\Theta^{(\al)}|$ in $\calT$ so we have
$\lambda(\al) = |\Theta_\phi^{(\al)}|$ for all $\al \in A_\f$.
Item (2) of Proposition \ref{delta-prop}  implies part (2) of Theorem~\ref{main-thm}.
If $\f'$ is another folding decomposition of another expanding irreducible 
train-track representative $f'$ of 
$\phi$, we get another distinguished factor $\Theta_{f'}$.   Since the cones $\calT_\f$ 
and $\calT_{\f'}$ must intersect, it follows by the minimality properties of $\Theta_\f$ and
$\Theta_{f'}$ in Proposition~\ref{delta-prop} that they are equal.
Item (3) of Proposition \ref{delta-prop} completes the proof. 
\end{proof}

\section{Example}\label{example-sec}
In this section, we  compute the cycle polynomial for an explicit example, and
compare the DKL and McMullen cones.

Consider the rose with four directed edges $a,b,c,d$ and the map:
\[ f= \left\{
\begin{array}{lll}
a \to & B \to & adb \\
c \to & D \to & cbd.
\end{array} \right.
\] 
Capital letters indicate the relavent edge in the opposite orientation 
to the chosen one.
 It is well known (e.g. Proposition 2.6 in  \cite{AlgomKfirRafi:small}) that if $f\from\tau \to \tau$ is 
a graph map, and $\tau$ is a graph with $2m$ directed edges, and for every edge $e$ 
of $\tau$, the path $f^{2m}(e)$ does not have back-tracking (see Definition~\ref{graph-def}),  then $f$ is a train-track 
map. One can verify that $f$ is a train-track map.
\begin{figure}[htbp] 
   \centering
   \includegraphics[width=1.2in]{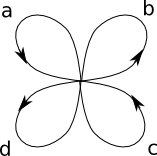} 
   \caption{Four petal rose with directed edges.}
   \label{Rose4directed-fig}
\end{figure}

The train-track transition matrix is given by
$$
M_f = 
\left [
\begin{array}{cccc}
0 & 1 & 0 & 0\\
1 & 1 & 0 & 1\\
0 & 0 & 0 & 1\\
0 & 1 & 1 & 1
\end{array}
\right ].
$$
The associated digraph is shown in figure~\ref{Ydigraph-fig}.
\begin{figure}[htbp] 
   \centering
   \includegraphics[width=1.5in]{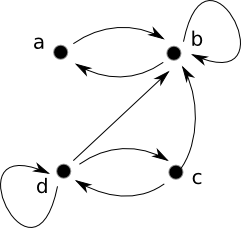} 
   \caption{Digraph associated to the train-track map $f$.}
   \label{Ydigraph-fig}
\end{figure}

The matrix $M_f$ is non-negative and $M_f^3$ is positive.  Thus $M_f$ is a Perron-Frobenius matrix
and $f$ is  a PF train-track map, hence an expanding irreducible train-track map
(see Section~\ref{freegroupautomorphisms-sec} for definitions).
By Theorem~\ref{DKL-thm}, $\alpha_\phi$ has an open cone neighborhood,
the DKL cone
$\calA_\f \subset \Hom(\Gamma;\R)$ such that the primitive integral elements of $\calA_\f$ correspond to
free group automorphisms that can be represented by expanding irreducible train-track maps.

\begin{remark}
The outer automorphism $\phi$ represented by $f$ is reducible. Consider the free factor   $\langle bA, ad, ac \rangle$ then 
\[ f(bA) = BDAb, \quad f(ad) = BDBC = aBDAaBCA, \quad f(ac) = BD = bAad. \]
Therefore this factor is invariant up to conjugacy. 
Thus $\phi$ is reducible, but  $f$ is a PF train-track map, and hence it is expanding and irreducible.  Thus we can
apply both Theorem~\ref{DKL1-thm} and Theorem~\ref{main-thm}.
\end{remark}

Identifying the fundamental group of the rose with $F_4$ we choose the basis $a,b,c,d$ of $F_4$. The free-by-cyclic group corresponding to $[f_*]$ has the presentation:  
\[ \Gamma = \langle a,b,c,d,s' \mid a^{s'} = B, b^{s'} = BDA, c^{s'} = D, d^{s'} = DBC \rangle. \]
Let  $G = \Gamma ^{\text{ab}}$ and for $w \in \Gamma$ we denote by $[w]$ its image in $G$. Then 
\[ [a] = -[b] = [d] = -[c].\]
Thus $G = \ZZ^2 = \langle t,s \rangle$ where $t=[a]$ and $s = [s']$. We decompose $f$ into four folds
\[ \tau=\tau_0 \xrightarrow{f_1} \tau_1 
\xrightarrow{f_2} \tau_2
\xrightarrow{f_3} \tau_3
\xrightarrow{f_4} \tau_4 \cong \tau,\] 
where all the graphs $\tau_i$ are roses with 
four petals. $f_1$ folds all of $a$ with the first third of $b$, to the edge $a_1$ of 
$\tau_1$, the other edges will be denoted $b_1$, $c_1$, $d_1$. $f_1$ folds the edge 
$c_1$ with the first third of the edge $d_1$. With the same notation scheme, $f_2$ folds 
the edge $c_2$ with half of the edge $b_2$ and $f_3$ folds the edge $a_3$ with half 
of the edge $d_3$. Figure \ref{foldDiagramX} shows the folded mapping torus 
$X_\f$ for this folding sequence. 

\begin{figure}[hbt]
\begin{center}
\includegraphics[width=2.5in] {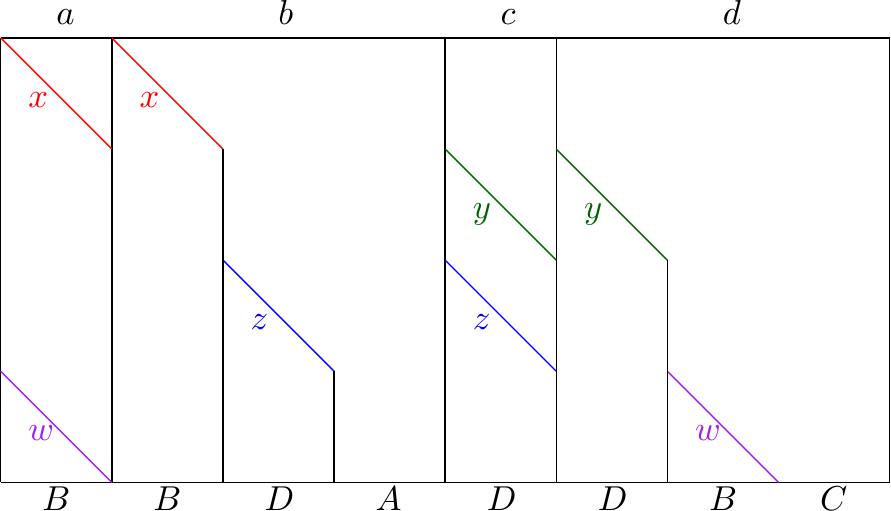}
\caption{\label{foldDiagramX}The complex $X$}
\end{center}
\end{figure}

The cell structure $\mathfrak{C}_\f$ 
 has 4 vertices, 8 edges: $s_1, s_2, s_3, s_4, x,y,z,w$, 
and four 2-cells: $c_x, c_y, c_z, c_w$. 
The 2-cells are sketched in Figure \ref{discs}. 

\begin{figure}[htb]
\begin{center}
\includegraphics[width=3.5in]{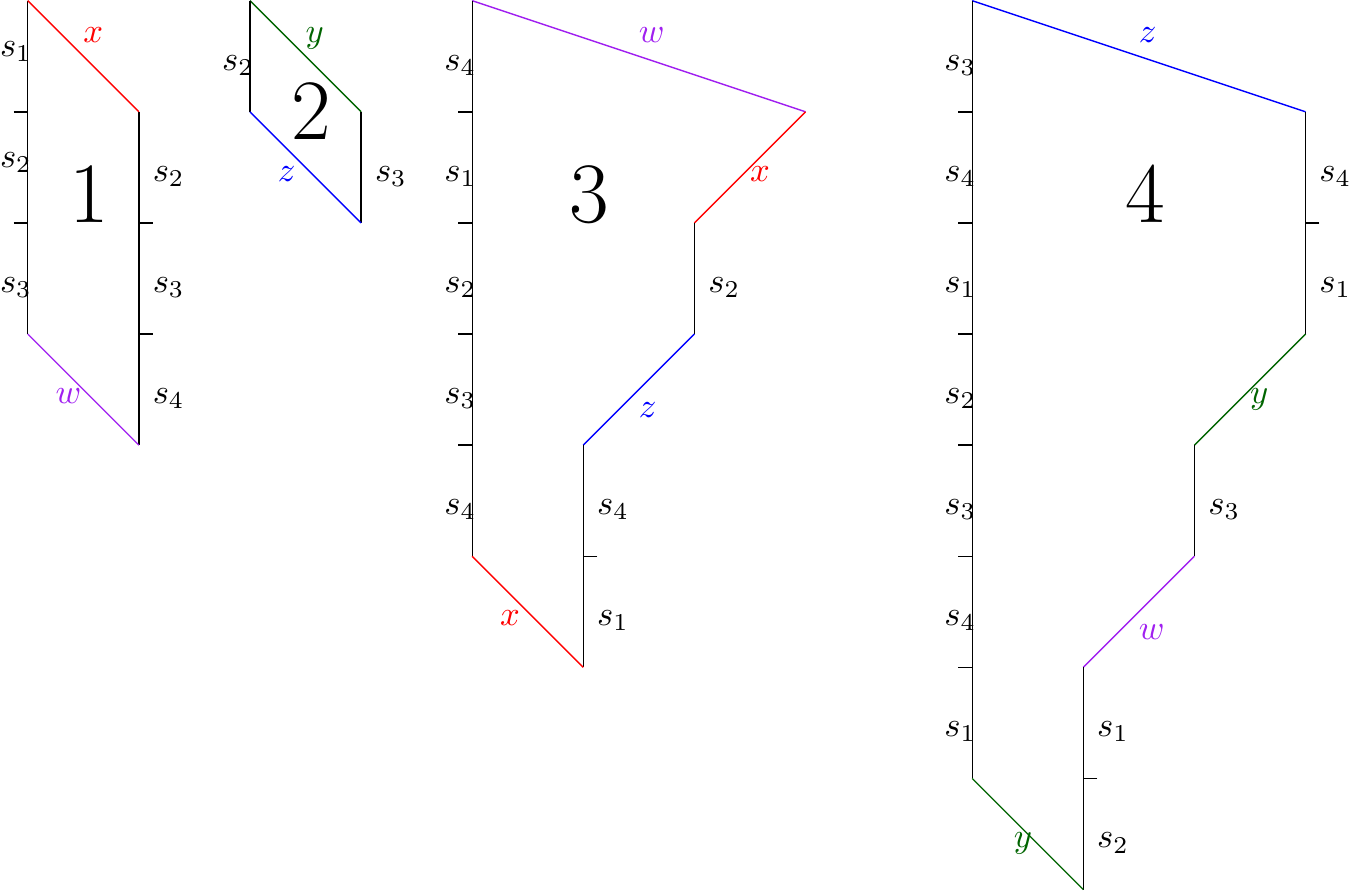}
\caption{\label{discs}The discs in $X$}
\end{center}
\end{figure}

Let $C_1$ be the free abelian group generated by the edges of $X_\f$, and 
let $F$ be the maximal tree consisting of the edges $s_1, s_2, s_3$, then  
$Z_1 \subset  C_1$ is generated by $x,y,z,w$ and $s_1+s_2+s_3+s_4$. 
The quotient homomorphism $\nu \from Z_1 \to G$ is given by collapsing the maximal tree and considering the relations given by the two cells. 
The map is given by $\nu(s_1 + s_2 + s_3 + s_4) = s$ and 
\[ 
\nu(x)=t \quad \nu(y) =  \nu(z) = -t \quad \nu(w) = t+ s. \]

\begin{figure}[ht]
\begin{center}
\includegraphics[width=4in]{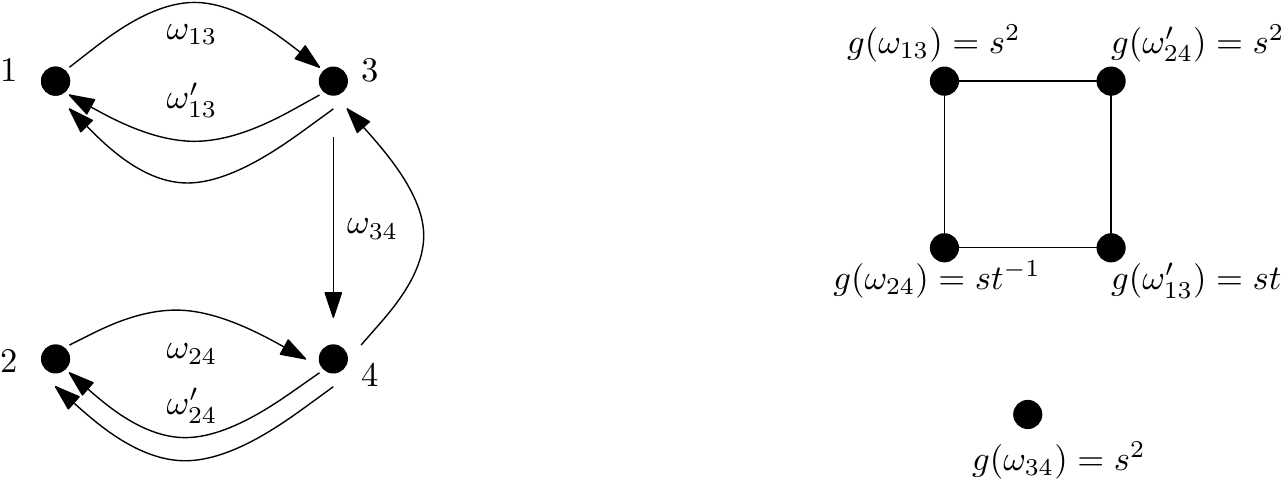}
\caption{\label{DigraphAndCycleComp}The vertical diagraph is on the left and the labeled cycle complex on the right.}
\end{center}
\end{figure}

The dual digraph $D$ to $X$ is shown on the left of Figure~\ref{DigraphAndCycleComp}. There are five cycles: $\omega_{13}$ and $\omega_{13}'$ the two 
distinct cycles containing $1$ and $3$, $\omega_{24}$ and $\omega_{24}'$ the two distinct 
cycles containing $2$ and $4$, and $\omega_{34}$ is the cycle containing $3$ and $4$. 
The cycle complex is  shown on the right of Figure \ref{DigraphAndCycleComp}.

\[ \begin{array}{ll}
\theta_f & = 1 - (s^{-2} + s^{-1} t^{-1} + s^{-2} + s^{-1} t + s^{-2} ) + ( s^{-3} t + s^{-2} + s^{-3} t^{-1} + s^{-4}) \\[0.2 cm]
& = 1 + s^{-4} -2s^{-2} - s^{-1}t^{-1} - s^{-1}t + s^{-3}t + s^{-3}t^{-1}
\end{array}
\]
Note that $\Theta_\phi$ might be a proper factor of this polynomial. However, for the
sake of computing the support cone (and the dilatations of $\phi_\al$ for different $\al \in \calA_\f$) we may use $\theta_f$. 

\noindent
\textbf{Computing the McMullen cone:} In order to simplify notation, for 
$\al \in \Hom(G,\RR)$ and $g \in G$ we denote $g^\al = \al(g)$. 
The cone $\calT_\phi$ in $H^1(G,\RR)$ is given by
\[ \begin{array}{ll}
\calT_\phi & = \{ \al \in \Hom(G,\RR) \mid  g^\al < 0^{\al} \text{ for all }g \in \text{Supp}(\theta_f) \} \\[0.2 cm]
 & = \left\{ \al \in \Hom(G,\RR) \left|  
\begin{array}{l}
(-4s)^\al, (-2s)^\al, (-s-t)^\al < 0 \\
(-s+t)^\al, (-3s+t)^\al, (-3s-t)^{\al}<0 
\end{array}
\right\}. \right.
\end{array}.
\]
Therefore, the McMullen cone is 
\begin{equation}\label{eqConeT} \calT_\phi = \{  \al \in \Hom(G,\RR) \mid s^\al>0 \quad \text{and} \quad |t^\al|< s^\al \}. \end{equation}
\begin{figure}[hb]
\begin{center}
\includegraphics[width=2in]{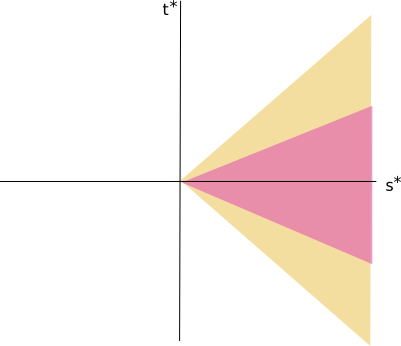}
\caption{\label{cones-fig}The McMullen cone $\calT$ (outer)  and DKL cone $\calA_\f$ (inner).}
\end{center}
\end{figure}
\noindent
\textbf{Computing the DKL cone:}
We now compute the DKL cone $\calA_\f$. A cocycle $\sf a$ represents an element  in $\al \in \calA_\f$ if it evaluates positively on all edges in $X_\f$. We use the notation: ${\sf a}(e) = e^{\sf a}$. Thus for ${\sf a}$ a positive cocycle: we have \[s_1^{\sf a}, s_2^{\sf a}, s_3^{\sf a} > 0\] and 
\[ s_4^{\sf a} > 0 \implies s^{\sf a}-s_1^{\sf a}+s_2^{\sf a}+s_3^{\sf a}>0 \implies s^{\sf a}>s_1^{\sf a}+s_2^{\sf a}+s_3^{\sf a}>0.\]
Now by considering the cell structure given by all edges in Figure~\ref{cones-fig} and recalling that $[a]=[d]=t$ and $[b]=[c]=-t$ we have: 
\[x = t+s_1 \quad w = t + s_4 \quad y = s_2 - t \quad z = s_3 -t. \]
The diagonal edges $x,w$ give us: \[ 0<x^{\sf a}=t^{\sf a}+ s_1^{\sf a} \text{ and } 0<w^{\sf a}=t^{\sf a}+s_4^{\sf a},\] 
so 
\[t^\al> - \frac{s_1^{\sf a} + s_4^{\sf a}}{2}> -\frac{s^\al}{2}.\]
The other diagonal edges give us
\[ 0<z^{\sf a} = s_3^{\sf a} - t^{\sf a} \text{ and } 0 < y^{\sf a} = s_2^{\sf a} - t^{\sf a}, \] hence 
\[ t^{\al} < \frac{s_2^{\sf a} + s_3^{\sf a} }{2} < \frac{s^\al}{2}. \]
We obtain the cone:
\begin{equation}\label{eqConeA} \left\{ s^\al>0 ~ \text{ and } ~ |t^\al |< \frac{s^\al}{2} \right\}.
 \end{equation}
It is not hard to see that if $\al$ is in this cone we may find a positive cocycle ${\sf a}$ representing $\al$. Therefore $\calA_\f$ is equal to the cone in (\ref{eqConeA}) and is strictly contained in the cone $\calT_\phi$ (see  (\ref{eqConeT}) and Figure~\ref{cones-fig}).

\end{document}